\documentclass[11pt]{amsart} 
\usepackage{amsthm,amsbsy,amsmath,amssymb,amscd,amsfonts,array,mathrsfs,verbatim,enumerate,xypic}
\usepackage[all]{xy}
\xyoption{arc} 

\newtheorem{thm}{Theorem}[section]
\newtheorem{prop}[thm]{Proposition}     
\newtheorem{lem}[thm]{Lemma}
\newtheorem{corollary}[thm]{Corollary}
\newtheorem{conjecture}[thm]{Conjecture}

\theoremstyle{definition}

\newtheorem{defn}{Definition}[thm]

\newtheorem{example}{Example}[thm] 
\newtheorem{rem}{Remark}[thm]

\DeclareFontFamily{OT1}{rsfs}{}
\DeclareFontShape{OT1}{rsfs}{n}{it}{<-> rsfs10}{}
\DeclareMathAlphabet{\curly}{OT1}{rsfs}{n}{it}

\newcommand{\Ab}{{\bf Ab}}   
\newcommand{\Sch}{{\bf Sch}} 

\renewcommand{\AA}{\mathbb{A}}  
\newcommand{\II}{\mathbb{I}} 
\newcommand{\ZZ}{\mathbb{Z}} 
\newcommand{\QQ}{\mathbb{Q}} 
\newcommand{\NN}{\mathbb{N}} 
\newcommand{\CC}{\mathbb{C}} 
\newcommand{\LL}{\mathbb{L}} 
\newcommand{\PP}{\mathbb{P}} 

\newcommand{\slot}{ \hspace{0.05in} {\rm \_} \hspace{0.05in} } 
\newcommand{\m}{\mathfrak{m}}         

\newcommand{\bR}{\operatorname{\bf R}}
\newcommand{\bL}{\operatorname{\bf L}}

\newcommand{\C}{\mathscr{C}} 
\renewcommand{\O}{\mathcal O} 

\newcommand{\ov}{\overline}

\newcommand{\Hom}{\operatorname{Hom}}   
\newcommand{\Ext}{\operatorname{Ext}}
\renewcommand{\H}{\operatorname{H}}
\newcommand{\R}{\operatorname{R}}

\newcommand{\Aut}{\operatorname{Aut}}

\newcommand{\Tor}{\curly Tor}             
\newcommand{\sHom}{\curly Hom}            
\newcommand{\sExt}{\curly Ext}            
\newcommand{\Supp}{\operatorname{Supp}}  
\newcommand{\gr}{\operatorname{gr}}

\newcommand{\Ker}{\operatorname{Ker}}
\newcommand{\Cok}{\operatorname{Cok}}

\newcommand{\Sym}{\operatorname{Sym}}    

\newcommand{\Id}{\operatorname{Id}}
\newcommand{\Spec}{\operatorname{Spec}}

\newcommand{\Quot}{\operatorname{Quot}}
\newcommand{\Hilb}{\operatorname{Hilb}}
\newcommand{\ch}{\operatorname{ch}}
\newcommand{\td}{\operatorname{td}}
\newcommand{\rk}{\operatorname{rk}}
\newcommand{\aj}{a}                     
\newcommand{\Gr}{\operatorname{Gr}}     

\newcommand{\Jac}{\operatorname{Pic}}

\begin{document}

\author{W.~D.~Gillam}
\address{Department of Mathematics, Brown University, Providence, RI 02912}
\email{wgillam@math.brown.edu}
\date{\today}
\title{Maximal subbundles, Quot schemes, \\ and curve counting}

\begin{abstract} Let $E$ be a rank $2$, degree $d$ vector bundle over a genus $g$ curve $C$.  The loci of stable pairs on $E$ in class $2[C]$ fixed by the scaling action are expressed as products of $\Quot$ schemes.  Using virtual localization, the stable pairs invariants of $E$ are related to the virtual intersection theory of $\Quot E$.  The latter theory is extensively discussed for an $E$ of arbitrary rank; the tautological ring of $\Quot E$ is defined and is computed on the locus parameterizing rank one subsheaves.  In case $E$ has rank $2$, $d$ and $g$ have opposite parity, and $E$ is sufficiently generic, it is known that $E$ has exactly $2^g$ line subbundles of maximal degree.  Doubling the zero section along such a subbundle gives a curve in the total space of $E$ in class $2[C]$.  We relate this count of maximal subbundles with stable pairs/Donaldson-Thomas theory on the total space of $E$.  This endows the residue invariants of $E$ with enumerative significance: they actually \emph{count} curves in $E$.  \end{abstract}

\maketitle

\section{Introduction} \label{section:introduction}

This note is concerned with the Gromov-Witten (GW), Donaldson-Thomas (DT), and Pandharipande-Thomas stable pairs (PT) residue invariants of the total space of a rank $2$ bundle $E$ over a smooth proper curve $C$ (see \S\ref{section:curvecounting} for a brief review).  The latter invariants are well-understood from a computational perspective.  In this paper I attempt to shed some light on the enumerative significance of such invariants and to explain the relationship between sheaf theoretic curve counting on $E$ and the virtual intersection theory of $\Quot$ schemes of symmetric products of $E$.  In particular, in \S\ref{section:maximalsubbundles}, I explain how to relate the ``count" of maximal subbundles of $E$ (which belongs properly to the theory of stable bundles on curves) with the DT/PT theory of $E$.  

Recall that the $\Quot$ scheme of a trivial rank $n$ bundle on a smooth proper curve $C$ may be viewed as a compactification of the space of maps from $C$ to a Grassmannian.  The relationship between the virtual intersection theory of this $\Quot$ scheme and various ``Gromov invariants" has been studied by many authors \cite{PR}, \cite{Ber}, \cite{BDW}, \cite{MO}, culminating in the theory of \emph{stable quotients} \cite{MOP} where the curve $C$ is also allowed to vary, as it is in GW theory.  The relationship between GW invariants of Grassmannians and counts of subbundles of maximal degree has also been studied by several authors \cite{Hol}, \cite{LN}, \cite{OT}.  These connections, however, are rather difficult to make, and one must take a circuitous route to link maximal subbundle counts to GW invariants.  Since counting maximal subbundles is an inherently sheaf-theoretic problem, it is reasonable to suspect that the most direct connections with curve counting should be made through the sheaf theoretic curve counting theories of Donaldson-Thomas \cite{Tho}, \cite{MNOP} and Pandharipande-Thomas \cite{PT}.

In general, the relationship between the PT theory of $E$ and virtual intersection theory on the various $\Quot \Sym^n E$ is quite subtle; we will treat the general case in a separate paper \cite{Gil2}.  In the present article, we will focus on the case of PT residue invariants in homology class $2[C]$ (twice the class of the zero section of $E$), where the most complete results can be obtained.  In particular, we will see that PT residue invariants of $E$ in class $2[C]$ are completely determined by virtual intersection numbers on $\Quot \O_C$ (symmetric products of $C$) and on the $\Quot$ scheme $\Quot^1 E$ parameterizing rank one subsheaves of $E$.  Our methods could even be used to compute the \emph{full} PT theory of $E$ (including descendent invariants involving odd cohomology classes from $C$) in class $2[C]$, though we do not provide the details here.

Once we have in hand the relationship between PT theory and virtual intersection theory of $\Quot$ schemes described above, we will be all the more interested in the latter.  In \S\ref{section:tautologicalclasses}, we suggest packaging this theory into a ``tautological ring."  In the course of proving the Vafa-Intriligator formula, Marian and Oprea \cite{MO} explained that this entire theory can be reduced, via virtual localization, to intersection theory on symmetric products of $C$, hence it can be treated as ``known."  On the other hand, it is quite painful in practice to write down manageable formulas for such invariants and it seemed to me to be overkill to appeal to the general results of \cite{MO} for the invariants actually needed in our study.  

Instead, we give (\S\ref{section:quotschemes}) a direct computation of the virtual intersection theory of the $\Quot$ scheme $\Quot^1 V$ parameterizing rank one subsheaves of a vector bundle $V$ on $C$ as follows.  The universal such subbundle $S$ is a line bundle on $\Quot^1 V \times C$, hence its dual yields a map $S^\lor : \Quot^1 V \to \Jac C$.  In case $V=\O_C$ this is the ``usual" map $\Sym^d C \to \Jac^d C$.  Just as in the case of the ``usual" map, the map $S^\lor$ is a projective space bundle when $S^\lor$ has sufficiently large degree $d$, and, as in the ``usual" case, one can compute the desired (virtual) intersection theory by descending induction on $d$.  The only new ingredient is the use of the virtual class; otherwise the computation is not significantly different from what MacDonald does in \cite{Mac}.

We include a review of the DT/GW/PT correspondence for the residue invariants of $E$ in \S\ref{section:correspondence}.  In principle, our computations could be used to verify this explicitly in degree $2[C]$, though we content ourselves with checking some special cases in \S\ref{section:computations}.

\subsection*{Acknowledgements} Much of this note is based on conversations with Matt Deland and Joe Ross which took place at Columbia in the spring of 2009.  I thank Ben Weiland for helpful discussions and Davesh Maulik for spurring my interest in PT theory.  This research was partially supported by an NSF Postdoctoral Fellowship.

\subsection*{Conventions} We work over the complex numbers throughout.  All schemes considered are disjoint unions of schemes of finite type over $\CC$.  We write $\Sch$ for the category of such schemes.  Set $T := \mathbb{G}_m = \Spec \CC[t,t^{-1}]$.  We will often consider an affine morphism $f : Y \to X$, in which case a $T$ action on $Y$ making $f$ equivariant for the trivial $T$ action on $X$ is a $\ZZ$-grading on $f_* \O_Y$ as an $\O_X$ algebra (so the corresponding direct sum decomposition is in the category of $\O_X$ modules).  A $T$ equivariant $\O_Y$ module is then the same thing as a graded $f_* \O_Y$ module.  Throughout, if $\pi : E \to X$ is a vector bundle on a scheme $X$, we use the same letter to denote its (locally free coherent) sheaf of sections, so $E = \Spec_X \Sym^* E^\lor$.  The sheaf $\Sym^* E^\lor$ has an obvious $\ZZ$-grading supported in nonnegative degrees; we call the corresponding $T$ action the \emph{scaling action}.  If $Z \subseteq E$ is a subscheme, then by abuse of notation, we will also use $\pi$ to denote the restriction of $\pi$ to $Z$.

When $T$ acts trivially on $X$, the $T$ equivariant cohomology $\H^*_T(X)$ is identified with $\H^*(X)[t]$; a $T$ equivariant vector bundle $V$ on $X$ decomposes into eigensubbundles $V=\oplus_n V_n$ where $T$ acts on $V_n$ through the composition of the character $\lambda \mapsto \lambda^n$ and the scaling action.  For $n \neq 0$, the $T$ equivariant Euler class $e_T(V_n)$ is invertible in the localized equivariant cohomology $\H^*_T(X)_t = \H^*(X)[t,t^{-1}]$.  The eigensubbundle decomposition induces an analogous decomposition $K_T(X)=\oplus_{\ZZ} K(X)$ of the $T$ equivariant vector bundle Grothendieck group of $X$ and $e_T$ descends to a group homomorphism \begin{eqnarray*} e_T : \oplus_{\ZZ \setminus \{ 0 \} } K(X) & \to & (\H^*_T(X)^\times, \cdot) \\ V-W & \mapsto & e_T(V)e_T(W)^{-1} \end{eqnarray*} which plays a fundamental role in (virtual) localization.

Tensor products are usually denoted by juxtaposition and exponents.  Notation for pullback of sheaves is often omitted, so $E$ sometimes means $f^* E$ for a morphism $f$ which should be clear from context.

\section{Quot schemes} \label{section:quotschemes} Let $C$ be a smooth projective curve, $E$ a vector bundle on $C$.  Recall \cite{Gro} that the \emph{quotient scheme} $\Quot E \in \Sch$ represents the presheaf on $\Sch$ which associates to $Y$ the set of exact sequences $$0 \to S \to \pi_2^*E \to Q \to 0$$ on $Y \times C$ with $Q$ (hence $S$) flat over $Y$ (up to isomorphism).  In particular, a point of $\Quot E$ is a SES $$0 \to S \to E \to Q \to 0$$ of sheaves on $C$.  The restriction maps for this presheaf along $f : Y \to Y'$ are defined by pulling back via $f \times \Id_C$.  The identity map of $\Quot E$ corresponds to an exact sequence $$0 \to S \to \pi_2^*E \to Q \to 0$$ on $\Quot E \times C$ called the \emph{universal exact sequence}.  Let $\Quot^{r,e} E$ denote the component of $\Quot E$ where $S$ has rank $r$ and degree $e$.  Note that $r,e$ determine the Hilbert polynomial of $Q$ via Riemann-Roch.  $\Quot^{r,e}$ is a projective scheme.  When $r=1$ we drop it from the notation and refer to $\Quot^e E$ as the \emph{rank one Quot scheme}.   This will be the main object of interest in the later sections of the paper.  A subsheaf $S$ of a vector bundle on a curve is torsion free since it is contained in a torsion free sheaf, hence it is locally free because $\O_C$ is a sheaf of PIDs.  By flatness, the universal $S$ is a locally free sheaf on $\Quot E \times C$.

\subsection{Abel-Jacobi maps}  \label{section:abeljacobimaps} Let $E$ be a vector bundle on a smooth proper curve $C$.  The determinant $\land^r S$ of the universal $S$ is an invertible sheaf on $\Quot^{r,e} E \times C$ and hence defines an Abel-Jacobi map $$\aj_e : \Quot^{r,e} E \to \Jac^e C$$ to the Picard scheme of degree $e$ line bundles on $C$.  Indeed, recall that $\Jac C = \coprod_e \Jac^e C$ represents the presheaf (of abelian groups under tensor product) \begin{eqnarray*} \Sch^{\rm op} & \to & \Ab \\ Y & \mapsto & \frac{ \{ \; {\rm line \; bundles \; on } \; Y \times C \; \} }{ \{ \; {\rm line \; bundles \; pulled \; back \; from} \; Y \; \} } . \end{eqnarray*} Fix a point $x \in C$.  As in \cite{ACGH} IV.2, we fix, once and for all, a \emph{Poincar\'e} line bundle $L$ on $\Jac C \times C$ such that the restriction of $L$ to $\Jac C \times \{ x \}$ is trivial.  The pair $(\Jac C, L)$ has the following universal property, which characterizes it up to tensoring $L$ with a line bundle pulled back from $\Jac C$: for any scheme $Y$ and any line bundle $L'$ on $Y \times C$, there is a unique morphism $f: Y \to \Jac C$ such that $L' = (f \times \Id_C)^* L \otimes \pi_1^* M$ for some line bundle $M$ on $Y$.  Let $L_e$ denote the restriction of the Poincar\'e bundle $L$ to $\Jac^e C$; by definition of $\Jac^e C$, $L_e|_{\{ P \} \times C}$ is a degree $e$ line bundle on $C$ for each point $P$ of $\Jac^e C$.  Each of the components $\Jac^e C$ is non-canonically isomorphic to the torus $$\H^1(C,\O_C)/\H^1(C,\ZZ) = \Ker ( c_1 : \H^1(C,\O_C^*) \to \H^2(C,\ZZ) ),$$ so its cohomology is an exterior algebra on $\H^1$: \begin{eqnarray*} \H^*(\Jac^e C,\ZZ) & = & \land^* \H^1(\Jac^e C,\ZZ) \\ & \cong & \land^* \ZZ^{2g}. \end{eqnarray*}

If $E=\O_C$, then $\aj_{-n}$ is just the usual Abel-Jacobi map \begin{eqnarray*}  \Sym^n C & \to & \Jac^n C \\ D  & \mapsto & \O(D),\end{eqnarray*} once we identify $\Sym^n C$ with $\Quot^{-n} E$ via $$D \mapsto  (\O_C(-D) \hookrightarrow \O_C).$$  The situation is very similar when $E$ is any invertible sheaf.   In the case of the \emph{rank one} Quot scheme $\Quot^e E$, the fiber of $\aj_e$ over a line bundle $L \in \Jac^e C$ is the projective space $$\PP \Hom(L,E)$$ because any nonzero map $L \to E$ must be injective.  Under the inclusion $$\PP \Hom(L,E) \hookrightarrow \Quot^e E,$$ the universal subbundle $S \to \pi_2^* E$ on $\Quot^e E \times C$ restricts to \begin{eqnarray*} \O(-1) \boxtimes L & \to & \pi_2^* E \\ f \otimes l & \mapsto & f(l) \end{eqnarray*} on $\PP \Hom(L,E) \times C$.

It is well-known that the Abel-Jacobi map $\aj_n : \Sym^n C \to \Jac^n C$ is a $\PP^{n-g}$ bundle for $n \geq 2g-2$.  In fact, a similar statement holds for $\Quot^e E$, as we now prove.

\begin{lem} For any coherent sheaf $E$ on a smooth proper curve $C$, there is an integer $d_0$ such that 
$\H^1(C, E \otimes L)=0$ for every line bundle $L$ on $C$ of degree at least $d_0$. \end{lem}

\begin{proof} This is an easy consequence of semicontinuity and quasi-compactness of the Jacobian $\Jac^0 C$.  Fix a degree one (hence ample) invertible sheaf $\O_C(1)$ on $C$.  Then for any fixed degree zero line bundle $L_0$ on $C$, there is a $d_0$ such that $\H^1(C,E \otimes L_0(n))=0$ for all $n \geq d_0$.  By semicontinuity, $\H^1(C,E \otimes L(n))=0$ for all $L$ in a neighborhood of $L_0$ in $\Jac^0 C$, so by quasi-compactness of $\Jac^0 C$ we can choose $d_0$ large enough that $\H^1(C,E \otimes L(n))=0$ for all $n \geq d_0$ for all degree zero line bundles $L$.  As $L$ ranges over all degree zero line bundles, $L(n)$ ranges over all degree $n$ line bundles. \end{proof}

\begin{thm} \label{thm:quotforsmalle} Let $E$ be a rank $N$, degree $d$ vector bundle on a smooth proper curve $C$, $\Quot^e E$ the Quot scheme of rank one, degree $e$ subsheaves of $E$.  For all sufficiently negative $e$, the Abel-Jacobi map $\aj_e : \Quot^e E \to \Jac^e C$ is a $\PP^{m-1}$ bundle for $m = (1-g)N+d - eN$. \end{thm}

\begin{proof} Let $L_e$ the universal degree $e$ line bundle on $\Jac^e E \times C$.  By flatness, cohomology and base change, the lemma, and Grauert's Criterion, $\R^1 \pi_{1*} (L_e^\lor \otimes \pi_2^* E)=0$ and $$V := \pi_{1*} (L_e^\lor \otimes \pi_2^* E)$$ is a vector bundle on $\Jac^e C$ of rank $m$ whenever $e$ is sufficiently negative.  Let $p : \PP V \to \Jac^e C$ be the projection, $i : \O_{\PP V}(-1) \to p^*V$ the tautological inclusion of bundles on $\PP V$.  The ``evaluation" map \begin{eqnarray*} e : \pi_1^* \O_{\PP V}(-1) \otimes (p \times \Id_C)^* L_e & \to & \pi_2^* E \\ f \otimes l & \mapsto & \iota f(l) \end{eqnarray*} is a map of bundles on $\PP V \times C$ with the correct discrete invariants, hence it defines a morphism $f : \PP V \to \Quot^e E$ with $(f \times \Id_C)^*j=e$, where $j : S \to \pi_2^*E$ is the universal subsheaf on $\Quot^e E \times C$.  I claim $f$ is an isomorphism.  

To construct its inverse, recall that, by the universal property of $\PP V$, lifting the Abel-Jacobi map $\aj_e : \Quot^e E \to \Jac^e E$ to a map $g : \Quot^e \to \PP V$ is the same thing as giving a rank one subbundle of the bundle $\aj_e^* V$ on $\Quot^e E$ (under the correspondence, this rank one subbundle will then be $g^*i$).  By flatness and cohomology and base change in the diagram $$\xymatrix@C+15pt{ \Quot^e E \times C \ar[r]^{\aj_e \times \Id_C} \ar[d]_{\pi_1} & \Jac^e C \times C \ar[d]^{\pi_1} \\ \Quot^e E \ar[r]^{\aj_e} & \Jac^e C, }$$ we have \begin{eqnarray*} \aj_e^* V & = & \aj_e^* \pi_{1*}(L_e^\lor \otimes \pi_2^* E) \\ &=& \pi_{1*}( (\aj_e \times \Id_C)^* (L_e^\lor \otimes \pi_2^*E)) .\end{eqnarray*}  By the universal property of $\Jac^e C$, there is a line bundle $M$ on $\Quot^e E$ such that \begin{eqnarray} \label{lb} (\aj_e \times \Id_C)^*L_e \otimes \pi_1^* M = S.\end{eqnarray}  We then have $$M^\lor \otimes \pi_{1*}((\aj_e \times \Id_C)^*L_e^\lor \otimes \pi_2^* E) = \pi_{1*}(S^\lor \otimes \pi_2^*E)$$ by the projection formula.  Tensoring the universal inclusion $j:S \hookrightarrow \pi_2^*E$ on $\Quot^e E \times C$ with $S^\lor$ and applying $\pi_{1*}$, yields a rank one subsheaf $$\O_{\Quot^e E} \hookrightarrow \pi_{1*}(S^\lor \otimes \pi_2^* E).$$  In fact it is a sub\emph{bundle}, because otherwise $\iota$ would fail to be injective at some point of $\Quot^e E$, which is absurd.  Tensoring with $M$, we obtain a rank one subbundle $$k := \pi_{1*}(j \otimes \Id_{S^\lor}) \otimes \Id_M : M \hookrightarrow a_e^* V,$$ and hence a morphism $g : \Quot^e E \to \PP V$ with the property $g^* i = k$.  

The proof is completed by checking that $fg$ is the identity of $\Quot^e E$ and $gf$ is the identity of $\PP V$, which is straightforward.  For example, checking that $fg=\Id$ amounts to showing that $(g \times \Id_C)^* j=j$.  We have \begin{eqnarray*} (g \times \Id_C)^* j &=& (g \times \Id_C)^*(f \times \Id_C)^* j \\ &=& (g \times \Id_C)^* e, \end{eqnarray*} and \begin{eqnarray*}  (g \times \Id_C)^* (\pi_1^* \O_{\PP V}(-1) \otimes (p \times \Id_C)^* L_e) &=& g^* \O_{\PP V}(-1) \otimes (a_e \times \Id_C)^* L_e \\ & = & \pi_1^*M \otimes (a_e \times \Id_C)^* L_e \\ & = & S.\end{eqnarray*}  Furthermore, if $s$ is a local section of $S$ corresponding to a local section $m \otimes l$ under the last isomorphism, then by definition of the evaluation morphism $e$, we compute \begin{eqnarray*} (g \times \Id_C)^*e(s) & = & (g^*i)(m)(l) \\ &=& k(m)(l) \\ &=& j(s), \end{eqnarray*} where the last equality is just unravelling the definition of $k$. \end{proof}

\begin{rem} Although the rank one Quot scheme $\Quot^e E$ is always smooth in genus zero (it is just the projective space $\PP \Hom(\O_{\PP^1}(e),E)$), it can be singular in general.  For example, if $C$ has genus $2$ and $E= \O_C \oplus \O_C$, then the fiber of the Abel-Jacobi map $$ a_{-2} : \Quot^{-2} E \to \Jac^{-2} C$$ over the dual $L^\lor$ of the $g^1_2$ on $C$ is $$\PP \Hom_C(L^\lor,E) = \PP( \H^0(C,L)^{\oplus 2}) \cong \PP^3$$ (the other fibers are $\PP^1$'s).   The singular locus of $\Quot^{-2} E$ is a quadric $\PP^1 \times \PP^1$ in this $\PP^3$ fiber corresponding to subsheaves $L^\lor \hookrightarrow E$ whose inclusion factors through an inclusion $\O_C \hookrightarrow E$.  The quotient of such a subsheaf is of the form $\O_C \oplus \O_{P+Q}$ and the tangent space at such a point is $4$ dimensional, with one dimensional obstruction space $$\Ext^1(L^\lor,\O_C \oplus \O_{P+Q}) = \H^1(C,L)$$ (the obstruction space is easily seen to vanish away from this locus). \end{rem}

\subsection{Stratification} \label{section:stratification}  Let $E$ be a vector bundle on a smooth proper curve $C$.  The $\Quot$ scheme $\Quot^{r,e} E$ can be stratified as follows.  Let $\Quot^{r,e}_0 E$ be the open subscheme of $\Quot^{r,e} E$ parameterizing sub\emph{bundles}.  Define a map \begin{eqnarray*} \iota_n : \Quot^{r,e+rn} E \times \Sym^n C & \to & \Quot^{r,e} E \\ (S \hookrightarrow E, D) & \mapsto & (S(-D) \hookrightarrow S \hookrightarrow E) \end{eqnarray*} (this obviously makes sense on the level of universal families).  Each map $\iota_n$ is a closed embedding.  We obtain a stratification $$\Quot^{r,e} E = \coprod_{n \geq 0} \iota_n[ \Quot^{r,e+rn}_0 E \times \Sym^n C]$$ by locally closed subschemes.  It is finite since $\Quot^{r,e+rn} E$ is empty for large $n$.  

In the rank one case, this stratification respects the Abel-Jacobi maps of \eqref{section:abeljacobimaps} in the sense that $a_e \iota_n = a_{e+n} \otimes a_{-n},$ where $\otimes$ is multiplication for the abelian variety $\Jac C$.  That is, the diagram $$\xymatrix{ \Quot^{e+n} E \times \Quot^{-n} \O_C \ar[d]_{a_{e+n} \times a_{-n}} \ar[r]^-{\iota_n} & \Quot^e E \ar[d]^{a_e} \\ \Jac^{e+n} C \times \Jac^{-n} C \ar[r]^-{\otimes} & \Jac^e C }$$ commutes.

\subsection{Obstruction theory}  \label{section:virtualclass} Let $E$ be a vector bundle on a smooth proper curve $C$ as usual and let $\pi: \Quot E \times C \to \Quot E$ be the projection.  The goal of this section is to show that there is a map \begin{eqnarray} \label{POTonQuot} \bR \sHom( \bR \pi_* \sHom(S,Q), \O_{\Quot E}) \to \LL_{\Quot E} \end{eqnarray} to the cotangent complex of $\Quot E$ (see \cite{Ill}) defining a perfect obstruction theory (POT) on $\Quot E$ in the sense of Behrend-Fantechi \cite{BF}.  Our proof basically follows the proof of Theorem~1 in \cite{MO}, but we work with the Behrend-Fantechi formalism as opposed to that of Li-Tian.\footnote{It is not obvious to me that anything in the work of \cite{MO} actually leads to a map like \eqref{POTonQuot}.}  Following \cite{MO}, we produce an embedding from $\Quot E$ into a smooth scheme\footnote{In fact, we will use the smooth locus of $\Quot E$, but the embedding will not be the identity map.} so that $\Quot E$ is realized as the zero section of a vector bundle with the expected rank.  We then simply identify the paradigm POT on this zero locus (whose virtual class is cap product with the Euler class of the bundle; see Section~6 in \cite{BF}) with the map \eqref{POTonQuot}.

Before carrying this out, I should mention that this is probably not ``the best" way to construct this POT.  A more intrinsic construction of this POT appears in my paper \cite{Gil1}.  The idea there is to use the ``reduced Atiyah class" of the universal quotient, which is a map $$S \to \pi_1^* \LL_{\Quot E} \otimes Q $$ in the derived category $D(\Quot E \times C)$.  Tensoring with the derived dual $Q^\lor$ of (the perfect complex) $Q$ and tracing, this yields a $D(\Quot E \times C)$ morphism $$Q^\lor \otimes^{\bL} S \to \pi_1^* \LL_{\Quot E}.$$ Using Serre duality for $\pi_1$, one can show that this is the same thing as a map \eqref{POTonQuot}.  It is shown in \cite{Gil1} that this is a POT.  It can be shown that the POT we will construct here coincides with the POT constructed in \cite{Gil1}, though we will omit this argument in the interest of brevity.  We have chosen to present this version of the construction of the POT to keep the paper self-contained.

Let $j : D \hookrightarrow C$ be an effective divisor in $C$ of degree $n$ and consider the embedding \begin{eqnarray*} \iota_D : \Quot^{r,e} E & \hookrightarrow & \Quot^{r,e-rn} E \\ (S \hookrightarrow E) & \mapsto & (S(-D) \hookrightarrow  E) \end{eqnarray*} obtained by restricting the map $\iota_n$ from the previous section to $$ \Quot^{r,e} E \times \{ D \} \cong \Quot^{r,e} E.$$  To save notation, set \begin{eqnarray*} Z & := & \Quot^{r,e} E \\ X & := & \Quot^{r,e-rn} E \end{eqnarray*} and write $0 \to \ov{S} \to \pi_2^* E \to \ov{Q} \to 0$ for the universal SES on $X \times C$.

By definition of $\iota_D$, we have a commutative diagram with exact rows $$\xymatrix{ 0 \ar[r] & S(-D) \ar@{^(->}[d] \ar[r] & \pi_2^* E \ar@{=}[d] \ar[r] & \ov{Q} \ar[d] \ar[r] & 0 \\ 0 \ar[r] & S \ar[r] & \pi_2^* E \ar[r] & Q \ar[r] & 0}$$ on $Z \times C$, where the top row is the pullback of the universal sequence on $X \times C$ via $\iota_D \times \Id_C$ (we should probably write $\ov{Q}|_{Z \times C}$ instead of $\ov{Q}$) and the bottom row is the universal sequence on $Z \times C$.  By abuse of notation, we also write $j$ for $j$ times the identity map of either $\Quot$ scheme, so the cokernel of the left vertical map can be written $j_* j^* S$.  By the Snake Lemma, the right vertical arrow fits into an exact sequence \begin{eqnarray} \label{quotientSES} 0 \to j_* j^* S \to \ov{Q} \to Q \to 0. \end{eqnarray}

\begin{lem} \label{lem:smoothness} For a divisor $D \subset C$ of sufficiently large degree $n$, the embedding $\iota_D$ factors through the smooth locus of $X = \Quot^{r,e-rn} E$. \end{lem}

\begin{proof}  By compactness of $Z = \Quot^{r,e} E$ and semicontinuity, we can choose a $D$ with degree large enough that \begin{eqnarray} \label{firstvanishing} \H^1(C, \sHom(S,Q)|_z (D))=0 \end{eqnarray} for every point $z$ of $Z$. 

We will use the formal criterion for smoothness, so we must show that we can solve the lifting problem $$\xymatrix{ Y_0 \ar@{^(->}[d] \ar[r] & Z \ar@{^(->}[r]^{\iota_D} & X \ar[d] \\ Y \ar[rr] \ar@{.>}[rru] & & \Spec \CC }$$  whenever $Y_0 \hookrightarrow Y$ is a closed embedding of affine schemes with square zero ideal $I$.  By basic obstruction theory of quotients, it is equivalent to show that the obstruction to such a lift in $$\Ext^1_{Y_0 \times C}(\ov{S}, \ov{Q} \otimes \pi_1^* I)$$ vanishes, so it suffices to prove that this whole vector space is zero.  The functor $\Hom_{Y_0 \times C}$ is the composition of three functors: $$\Hom_{Y_0 \times C} = \Gamma \circ \pi_{1*} \circ \sHom_{Y_0 \times C},$$ so to prove the vanishing of $$\Ext^1_{Y_0 \times C}(\ov{S},\ov{Q} \otimes \pi_1^* I) = \R^1 \Hom_{Y_0 \times C}(\ov{S},\ov{Q} \otimes \pi_1^* I),$$ it suffices (by the Grothendieck spectral sequence) to prove the vanishing of the three vector spaces $$\begin{array}{l} \H^1(Y_0, \pi_{1*} \sHom(\ov{S},\ov{Q} \otimes \pi_1^* I)  \\ \H^0(Y_0, \pi_{1*} \sExt^1(\ov{S}, \ov{Q} \otimes \pi_1^* I)) \\ \H^0(Y_0, \R^1 \pi_{1*} \sHom(\ov{S},\ov{Q}) \otimes I)). \end{array}$$  The first vanishes because $Y_0$ is affine and the second vanishes because $\ov{S}$ is locally free. 

For the third vanishing, we will show that $\R^1 \pi_{1*} \sHom(\ov{S},\ov{Q})=0$.  By cohomology and base change and the assumption that $Y_0 \to X$ factors through $\iota_D$, we might as well assume $Y_0=X$, so $\ov{S}=S(-D)$.  Using the vanishing \eqref{firstvanishing} together with cohomology and base change (and flatness of $\sHom(S,Q)$), we first observe that $$\R^1 \pi_{1*} \sHom(\ov{S},Q) = \R^1 \pi_{1*} \sHom(S,Q)(D) = 0.$$  Next, we tensor the sequence $\eqref{quotientSES}$ with $S^\lor(D)$ and apply $\pi_{1*}$ to get an exact sequence $$ \cdots \to \R^1 \pi_{1*} j_*j^*(S \otimes S^\lor(D)) \to \R^1 \pi_{1*} \sHom(\ov{S},\ov{Q}) \to \R^1 \pi_{1*} \sHom(\ov{S},Q) \to 0.$$  The term on the left vanishes because $\pi_1 j$ has relative dimension zero, so we obtain the desired vanishing. \end{proof}

The embedding $\iota_D$ identifies $Z$ with the closed subscheme of $X$ defined by the degeneracy locus where the map $$\pi_{1*}  \ov{S}|_{X \times D} \to \pi_{1*} \pi_2^* E|_{X \times D}$$ of vector bundles on $X$ is zero.  Indeed, $\iota_D$ identifies the presheaf $\Quot^{r,e} E$ with the subpresheaf of $\Quot^{r,e-rn} E$ which associates to a scheme $Y$ the set of $$0 \to \ov{S} \to \pi_2^* E \to \ov{Q} \to 0$$ in $(\Quot^{r,e-rn} E)(Y)$ such that the inclusion $j : \ov{S} \to \pi_2^* E$ factors (necessarily uniquely) as below. \begin{eqnarray} \label{factorization} \xymatrix{ \ov{S} \ar@{^(->}[r]^-j \ar[rd] & \pi_2^* E \\ & \ov{S}(Y \times D) \ar@{.>}[u]} \end{eqnarray}  If $\Spec (A \otimes B) \hookrightarrow X \times C$ is a basic affine open on which $\ov{S}$ and $\pi_2^* E$ are trivialized and $D \subset C$ is given by $f \in B$, then $j$ can be viewed as a monomorphism $$j : (A \otimes B)^r \to (A \otimes B)^N$$ of free $(A \otimes B)$ modules.  For any $A$ algebra $\phi : A \to C$, the morphism $\Spec C \to X$ factors through $\iota_D$ iff there is a factorization \eqref{factorization} with $Y=\Spec C$ iff there is a factorization  $$  \xymatrix{ (C \otimes B)^r \ar@{^(->}[r]^-{\ov{j}} \ar[rd]_{(1 \otimes f) \cdot \Id_r} & (C \otimes B)^N \\ & (C \otimes B)^r \ar@{.>}[u]} $$ of $C \otimes B$ module maps, where $\ov{j} = (\Spec \phi \times \Id_B)^*j.$  On the other hand, such a factorization exists iff $$\ov{j}[(C \otimes B)^r] \subseteq (1 \otimes f) \cdot (C \otimes B)^N$$ iff $\ov{j} \otimes_{C \otimes B} C \otimes B/f=0$.  But \begin{eqnarray*}\ov{j} \otimes_{C \otimes B} C \otimes B/f &=& \ov{j}|_{\Spec C \times D} \\ &=& (\Spec \phi \times \Id_D)^*(j|_{X \times D}),\end{eqnarray*} which vanishes iff $\Spec \phi \times \Id_D$ factors through the zero locus of $j|_{X \times D}$ iff $\Spec C \to X$ factors through the zero locus of $\pi_{1*} (j|_{X \times D})$.

\begin{example} If $E$ has rank one and degree $d$, then for any $n \in \ZZ_{\geq 0}$, there is an the identification $\Sym^n C \cong \Quot^{d-n} E$ which identifies $S^\lor$ with $\O(D)$, where $D \subseteq \Sym^n C \times C$ is the universal divisor.  When $D=Q$ is a point of $C$, the embedding \begin{eqnarray*} \iota_Q : \Sym^{n-1} C & \hookrightarrow & \Sym^n C \\ P_1+ \cdots+P_{n-1} & \to & P_1+\cdots+P_{n-1}+Q \end{eqnarray*} identifies $\Sym^{n-1} C \subset \Sym^n C$ with the zero locus of a section of the restriction of $\O(D)$ to $\Sym^n C \times \{Q\} = \Sym^n C$. \end{example}

Of course, $\pi_2^* E|_{X \times D} = \pi_2^* E|_D$ is a trivial bundle of rank $N$ equal to the rank of $E$, so  $\pi_{1*} \pi_2^* E|_D$ is a trivial bundle of rank $nN$ ($\pi_{1*} \O_{X \times D}$ is a free $\O_X$ module of rank $n = \deg D$).  The degeneracy locus $Z$ may therefore be viewed as the zero locus of a section of the bundle $V := (\pi_{1*} \ov{S}^\lor|_{X \times D})^N$ on $X$.

In general, if $Z$ is the zero locus of a vector bundle $V$ on a smooth scheme $X$ (smooth in a neighborhood of $Z$ is good enough), then there is a map $$[V^\lor|_Z \to \Omega_X|_Z] \to \LL_Z$$ (the complex on the left is supported in degrees $-1,0$) defining a POT on $Z$ whose associated virtual class is just the Chow class $e(V) \cap [X]$ supported on $Z$.  This is a special case of the discussion in Section~6 of \cite{BF}.  In the situation of this section, we have:

\begin{lem} The complexes $[V^\lor|_Z \to \Omega_X|_Z]$ and $\bR \sHom( \bR \pi_{1*} S^\lor Q, \O_Z)$ on $Z$ are naturally quasi-isomorphic. \end{lem}

\begin{proof} It is equivalent to prove that the dual complex $[TX|_Z \to V|_Z]$ is naturally quasi-isomorphic to $\bR \pi_{1*} S^\lor Q$.  By basic obstruction theory, we have an identification $TX = \pi_{1*} \ov{S}^\lor \ov{Q}$.  Following through this identification and the definition of the boundary map in the paradigm POT, we see that the boundary map $TX|_Z \to V$ is identified with $\pi_{1*}$ of the adjunction map $$\ov{S}^\lor \ov{Q} \to \ov{S}^\lor \ov{Q}|_{Z \times D}.$$

If we show that the vertical arrows in the natural diagram $$\xymatrix{ & \ov{S}^\lor \ov{Q} \ar[r] & \ov{S}^\lor \ov{Q}|_{Z \times D} \\ S^\lor S|_{Z \times D} \ar[r] & S^\lor \ov{Q} \ar[d] \ar[u] \\ & S^\lor Q  }$$ define quasi-isomorphisms between the complexes (on $Z \times C$) given by the rows, then the proof is completed by applying the functor $\bR \pi_{1*}$. To see that the bottom vertical arrow is a quasi-isomorphism, just tensor \eqref{quotientSES} with $S^\lor$. To say that the top vertical arrow is a quasi-isomorphism is equivalent to saying that the sequence \begin{eqnarray} \label{exactseq} 0 \to S^\lor S|_{Z \times D} \to S^\lor \ov{Q} \to \ov{S}^\lor \ov{Q} \to \ov{S}^\lor \ov{Q}|_{Z \times D} \to 0\end{eqnarray} is exact.  Tensoring the exact sequence $$0 \to S^\lor \to \ov{S}^\lor \to \ov{S}^\lor|_{Z \times D} \to 0 $$ with $\ov{Q}$ gives an exact sequence $$0 \to \Tor_1^{Z \times C}( \ov{S}^\lor |_{Z \times D} , \ov{Q}) \to S^\lor \ov{Q} \to \ov{S}^\lor \ov{Q} \to \ov{S}^\lor \ov{Q}|_{Z \times D} \to 0, $$ so we want to show that $S^\lor S|_{Z \times D}$ is identified with $\Tor_1^{Z \times C}( \ov{S}^\lor|_{Z \times D}, \ov{Q})$.  Tensoring the locally free resolution $$0 \to \ov{S} \to \pi_2^*E  \to \ov{Q} \to 0$$ (restricted from $X \times C$ to $Z \times C$) of $\ov{Q}$ with $\O_{Z \times D}$ gives an exact sequence $$0 \to \Tor_1^{Z \times C}(\O_{Z \times D}, \ov{Q}) \to \ov{S}|_{Z \times D} \to \pi_2^* E|_{Z \times D} \to \ov{Q}|_{Z \times D} \to 0 $$ of sheaves on $Z \times C$.  But by the definition of $Z$, the the map in the middle is zero, so the left and right maps are isomorphisms and we conclude $\Tor_1^{Z \times C}(\O_{Z \times D},\ov{Q})= \ov{S}|_{Z \times D}.$  Since $\ov{S}^\lor|_{Z \times D}$ is a locally free $\O_{Z \times D}$ module, we have \begin{eqnarray*} \Tor_1^{Z \times C}(\ov{S}^\lor|_{Z \times D}, \ov{Q}) & = & \Tor_1^{Z \times C}(\O_{Z \times D}, \ov{Q}) \otimes \ov{S}^\lor|_{Z \times D} \\ &=& \ov{S}^\lor \ov{S}|_{Z \times D} \\ & = & S^\lor S|_{Z \times D}. \end{eqnarray*} This proves that \eqref{exactseq} is exact, hence completes the proof.      \end{proof}

\begin{rem} One can show directly that this POT is independent of the choice of ``sufficiently positive" $j : D \hookrightarrow C$ by comparing two such $D_1,D_2$ with the inclusion of the sum $D_1+D_2$. \end{rem}

\subsection{Tautological classses}  \label{section:tautologicalclasses}  Let $\eta \in \H^2(C,\ZZ)$ be the fundamental class and fix a basis $\delta_1,\dots,\delta_{2g}$ for $\H^1(C,\ZZ)$ such that $\delta_i \delta_{g+i}=-\delta_{g+i}\delta_i=\eta$ for $i \in \{1,\dots,g \}$ and all other $\delta_i \delta_j$ are zero.    Let $S$ be the universal subbundle on $\Quot E \times C$.  Let $$c_i(S^\lor) = a_i \otimes 1 + \sum_{j=1}^{2g} b_{i,j} \otimes \delta_j + f_i \otimes \eta$$ be the K\"unneth decomposition of $c_i(S^\lor)$.  In the argot of \cite{Mar}, \cite{MO}, etc.\ the classes \begin{eqnarray*} a_i & \in & \H^{2i}(\Quot^{r,e} E) \\ b_{i,j} & \in & \H^{2i-1}(\Quot^{r,e} E) \\ f_i & \in & \H^{2i-2}(\Quot^{r,e} E)\end{eqnarray*} are called $a,b$ and $f$ \emph{classes}.   

\begin{defn} A \emph{tautological class} is an element of the subring $\ov{\R}^*$ of $\H^*(\Quot E)$ generated by the $a$, $b$, and $f$ classes.  The \emph{tautological ring} $\R^*$ is the quotient of $\ov{\R}^*$ by the ideal $$\{ \alpha \in \ov{\R}^* : \int_{[\Quot E]^{\rm vir}} \alpha \beta =0 \; {\rm for \;  all} \; \beta \in \ov{\R}^* \}.$$ \end{defn}

Unlike the actual cohomology ring $\H^*(\Quot E)$, the tautological ring is deformation invariant: it does not depend on the curve $C$ or the bundle $E$, except through their discrete invariants.  This follows from deformation-invariance properties of the virtual class and the fact that the $a$, $b$, and $f$ classes have globally defined analogues as $C,E$ vary in families which restrict to the given classes at points of the family.  The virtual intersection theory of $\Quot E$ is the study of the tautological ring, usually through calculations of integrals of tautological classes over the virtual class.  A typical result is the Vafa-Intriligator formula for integrals of polynomials in the $a$ classes \cite{MO} (for the case $E=\O_C^N$).

For later use, let us record some facts about the first Chern class $c_1(L_e) \in \H^2(\Jac^e C \times C)$ of the Poincar\'e bundle discussed on pages 334-335 of \cite{ACGH}.  Its K\"unneth decomposition looks like \begin{eqnarray} \label{c1ofL} c_1(L_e) &=& 0 \otimes 1 + \sum_{i=1}^{2g} b_i \otimes \delta_j + e \otimes \eta \end{eqnarray} (the $(2,0)$ K\"unneth component is zero because $L_e$ restricts trivially to $\Jac^e C \times \{ Q \}$).  Note that, by the universal property, any two components of $\Jac C$ are identified by tensoring $L$ with a line bundle pulled back from $C$; this only changes the $(0,2)$ K\"unneth component of $c_1$, so the $(1,1)$ K\"unneth components of $L_e$ and $L_{d}$ are identified under this isomorphism, so the $b_i$ are ``independent of $e$".  Recall that the $b_i$ form a basis for $\H^1(\Jac^e C, \ZZ) \cong \ZZ^{2g}$ which is oriented so that \begin{eqnarray} \label{topeven} \int_{\Jac^e C} b_1b_{g+1} b_2b_{g+2} \cdots b_g b_{2g} = 1.\end{eqnarray} 

To prove this (following \cite{ACGH}), one may assume $e$ is large.  Consider the divisor $\Delta + (e-1) C \times \{ Q \}$ in $C \times C$, where $\Delta$ is the diagonal.  By the universal property of $\Jac^e C$, the associated line bundle $$L' := \O_{C \times C}(\Delta+(e-1) C \times \{ Q \})$$ determines a morphism $f : C \to \Jac^e C$ so that $(f \times \Id_C)^*L_e$ differs from $L'$ only by tensoring with a line bundle pulled back from $C$ via $\pi_1$; this doesn't change the $(1,1)$ K\"unneth component of $c_1$, so we see that the $(1,1)$ K\"unneth component of $L_e$ pulls back via $f \times \Id_C$ to the $(1,1)$ K\"unneth component of $c_1(L')$, which is just the $(1,1)$ K\"unneth component of the diagonal $\Delta$: \begin{eqnarray*} \sum_{1=1}^{2g} f^* b_i \otimes \delta_i & = & \Delta_{1,1} \\ & = & \sum_{i=1}^g -\delta_{g+i} \otimes \delta_i + \sum_{i=1}^g \delta_i \otimes \delta_{g+i}. \end{eqnarray*}  To see that $f^*$ induces an isomorphism on $\H^1$, notice that it factors as a sequence of stratification maps (c.f.\ \eqref{section:stratification}, \eqref{section:virtualclass}) \begin{eqnarray*} \iota_Q : \Sym^n C & \to & \Sym^{n+1} C \\ P_1+\cdots+P_n & \mapsto & P_1+\cdots +P_n +Q \end{eqnarray*} (starting from $C=\Sym^1 C$) followed by the Abel-Jacobi map $\aj_e : \Sym^e C \to \Jac^e C$.  Since $e$ is large, $\aj_e$ is a $\PP^{e-g}$ bundle, hence $\aj_e^*$ is an isomorphism on $\H^1$, and each $\iota_Q$ has affine complement and is an isomorphism on $\H^1$ by the Lefschetz Hyperplane Theorem (c.f.\ (12.2) in \cite{Mac}).

Set $\theta := \sum_{i=1}^g b_i b_{g+1} \in \H^2(\Jac^e C)$, and let $\gamma = \sum_{i=1}^{2g} b_i \otimes \delta_i$ be the $(1,1)$ K\"unneth component of $c_1(L_e)$ as in \cite{ACGH}.  Then $\gamma^2 = - 2 \eta \theta$,  \begin{eqnarray} \label{thetatotheg} \int_{\Jac^e C} \theta^g = g! \int_{\Jac^e C} b_1b_{g+1}b_2b_{2+g} \cdots b_g b_{2g} = g!,\end{eqnarray}  and \begin{eqnarray} \label{chL} \ch L_e = 1+e \eta + \gamma-\eta \theta \end{eqnarray} (dropping notation for pullbacks).

Let us now focus on $c_1(S^\lor)=c_1(\land^r S^\lor)$, whose K\"unneth decomposition we will simply write as $$c_1(S^\lor) = a \otimes 1 - \sum_{i=1}^{2g} b_i \otimes \delta_i - e \otimes \eta$$ for the sake of this discussion.  As the notation intimates, $b_i$ is the pullback of $b_i \in \H^1(\Jac^e C)$ via the Abel-Jacobi map $\aj_e : \Quot^{r,e} E \to \Jac^e E$.  Indeed, by definition of $\aj_e$ (Section~\ref{section:abeljacobimaps}) via the universal property of $\Jac^e C$, $(\aj_e \times \Id_C)^*L_e$ is equal to $\land^r S$ up to tensoring with a line bundle pulled back from $\Quot^e E$; tensoring with such a line bundle does not effect the $(1,1)$ K\"unneth component of $c_1$.  Since $\H^*(\Jac C)$ is generated by $\H^1(\Jac C)$, all classes in $\H^*(\Quot E)$ pulled back via the Abel-Jacobi map are tautological.  We will usually drop $\aj_e^*$ from the notation, and simply write, e.g.\ $\theta \in \ov{\R}^2(\Quot^{r,e} E)$ for $\aj_e^* \theta$.  

The rest of this section is devoted to determining the tautological ring of the rank one Quot scheme.  The necessary intersection number formulae will be obtained as sequelae of the structure theory of the rank one Quot scheme from \eqref{section:abeljacobimaps} and the basic properties of the virtual class in \eqref{section:virtualclass}.  

A monomial in the $b_i$ equal to $b_{i_1}b_{g+i_1} \cdots b_{i_k}b_{g+i_k}$ for some distinct $i_1, \dots, i_k \in \{ 1, \dots, g \}$ in $\H^*(\Jac C)$ will be called \emph{even}, while a monomial in the $b_i$ not equal to (plus or minus) an even monomial will be called \emph{odd}.  The possible ``parities" of a (non-zero) product of monomials of given parity are: \begin{eqnarray*} ({\rm  odd }) ({\rm  even }) &=& {\rm \; odd \; } \\ ({\rm  even }) ({\rm  even }) &=& {\rm \; even  } \\ ({\rm  odd }) ({\rm  odd }) &=& {\rm \; odd \; or \; even } \end{eqnarray*}

\begin{thm} \label{thm:virtualintersectionnumbers} Let $E$ be a vector bundle of rank $N$ and degree $d$ on a smooth proper curve $C$.  Set $m := (1-g)N+d-Ne$ (as in Theorem~\ref{thm:quotforsmalle}).  For any $k \in \{0,\dots,g\}$ with $m-1+k \geq 0$, and any even monomial $b \in \H^{2g-2k}(\Quot^e E)$, $$\int_{[\Quot^e E]^{\rm vir}} a^{m-1+k} b = N^k,$$ while this integral vanishes for any odd $b$. \end{thm}

\begin{proof} We first prove the theorem when $e$ is sufficiently small.  Let $\pi_i$ be the projections to the factors of $\Jac^e C \times C$.  Since $e$ is small, $\bR^1 \pi_{1*} (L_e^\lor \otimes \pi_2^* E)=0$, $V := \pi_{1*} (L_e^\lor \otimes \pi_2^* E)$ is a vector bundle of rank $m$ on $\Jac^e C$, and $\Quot^e E = \PP V $ (see the proof of Theorem~\ref{thm:quotforsmalle}).  By GRR and \eqref{c1ofL} we find $\ch V = m-N \theta$.  By the proof of Theorem~\ref{thm:quotforsmalle}, we have $S=\O_{\PP V}(-1) \otimes (p \times \Id_C)^* L_e$, where $p : \PP V \to \Jac^e C$ is the projection.  Since the $(2,0)$ K\"unneth component of $L_e$ is zero, $a=c_1(\O_{\PP V}(1))$.  The Chern character formula for $V$ implies $c(V)=e^{-N \theta}$ (see page 336 in \cite{ACGH}), so by definition of Chern classes, we have the relation \begin{eqnarray} \label{relation} a^m &=& a^{m-1}(N\theta)-a^{m-2}\frac{(N \theta)^2}{2!}+ a^{m-3}\frac{(N \theta)^3}{3!} - \cdots \end{eqnarray}  in $\H^*(\PP V)=\ov{R}^*(\Quot^e E)$.

The proof in the small $e$ case can now be completed by induction on $k$.  When $k=0$, we compute \begin{eqnarray*} \int_{[\Quot^e E]^{\rm vir}} a^{m-1}b & = & \int_{\PP V} c_1(\O_{\PP V}(1))^{\rk V -1} b \\ &=& \int_{\Jac^e C} b, \end{eqnarray*} which, according to \eqref{topeven}, is indeed given by $N^0=1$ when $b$ is the unique even monomial of degree $2g$ and certainly vanishes when $b$ is odd (hence zero).  For the induction step, since $k>0$, we can use \eqref{relation} to write \begin{eqnarray} \label{step1} \int_{\PP V} a^{m-1+k} b &=& \int_{\PP V} a^{m-1+(k-1)}(N\theta)b - a^{m-1+(k-2)} \frac{(N \theta)^2}{2!} b +  \cdots .\end{eqnarray} Note $$\theta^j = j! \sum_{1 \leq i_1 < \cdots < i_j \leq g} b_{i_1}b_{g+i_1} \cdots b_{i_j} b_{g+i_j}$$ is $j!$ times the sum of all even monomials of degree $2j$.  If $b$ is odd, then by the rules for the parity of a product, $\theta^j b$ is a sum of odd monomials, hence \eqref{step1} vanishes by the induction hypothesis.  On the other hand, if $b$ is an even monomial of degree $2g-2k$, then $(\theta^j / j!)b$ is a sum of $\begin{pmatrix} k \\ j \end{pmatrix}$ even monomials of degree $2g-2(k-j)$ and we have $$\int_{\PP V} a^{m-1+(k-j)} \frac{\theta^j}{j!}b = \begin{pmatrix} k \\ j \end{pmatrix} r^{k-j}$$ by the induction hypothesis, hence \begin{eqnarray*} \eqref{step1} & = & N^k \left ( \begin{pmatrix} k \\ 1 \end{pmatrix} - \begin{pmatrix} k \\ 2 \end{pmatrix} + \begin{pmatrix} k \\ 3 \end{pmatrix} - \cdots \right ) \\ &=& N^k, \end{eqnarray*} as claimed.

For a general $e$, choose a divisor $D \subset C$ of sufficiently large degree $n$ and consider the embedding $$\iota_D : \Quot^e E \hookrightarrow \Quot^{e-n} E $$ of Section~\ref{section:virtualclass}.  Following the notation of Section~\ref{section:virtualclass}, let $Z = \Quot^e E$, $X = \Quot^{e-n} E$, and let $S$, $\ov{S}$ be the universal subsheaves on $Z \times C$ and $X \times C$, respectively.  As in Section~\ref{section:virtualclass}, we have $S = \ov{S}(D)|_{Z \times C}$.  The twist by $\pi_2^* \O_C(D)$ doesn't change the $(2,0)$ K\"unneth component of $c_1$, so we have $a = \ov{a}|_Z$.  Let $\ov{S}_D^\lor := \pi_{1*} (\ov{S}^\lor|_{X \times D})$.  Evidently $\ov{S}_D^\lor$ is a vector bundle on $X$ with rank $n$ and $c_n( \ov{S}_D^\lor ) = \ov{a}^n$.  We saw in Section~\ref{section:virtualclass} that $$[Z]^{\rm vir} = c_n( \ov{S}_D^\lor )^r \cap [X] $$ (we are assuming that $n$ is large enough that $X$ is a $\PP^{\ov{m}-1}$ bundle over $\Jac^e C$, so in particular smooth of the expected dimension).  Note that \begin{eqnarray*} \ov{m} & = & (1-g)N+d-(e-n)N \\ & = & m+nN. \end{eqnarray*}  Since we know the desired formula holds on $X$, the proof is completed by the following computation: \begin{eqnarray*} \int_{[Z]^{\rm vir}} a^{m-1+i} b &=& \int_{X} \ov{a}^{m-1+i} b c_n( \ov{S}_D^\lor )^r \\ & = & \int_{X} \ov{a}^{\ov{m}-1+i} b. \end{eqnarray*} \end{proof}

The formula \begin{eqnarray} \label{powersofatheta} \int_{[\Quot^e E]^{\rm vir}} a^{m-1+k} \theta^{g-k} & = &  \frac{N^k g!}{k!} \end{eqnarray} is an immediate consequence.  The following corollary is also an important special case of the theorem.

\begin{corollary} \label{corollary:zerodimensionalvirtualclass} Let $E$ be a vector bundle of rank $N$ and degree $d$ on a smooth proper curve $C$ and suppose $e \in \ZZ$ satisfies $$(1-g)(N-1)+d-Ne=0$$ (i.e.\ the expected dimension of $\Quot^e E$ is zero).  Then $$\int_{[\Quot^e E]^{\rm vir}} 1 = N^g.$$\end{corollary}

Notice that the virtual intersection theory of $\Quot^e E$ has little dependence on the rank of $E$.  In case $N=1$, $\Quot^e E=\Sym^n C$ (for an appropriate $n$) and the tautological ring $R^*(\Quot^e E)$ is just the usual cohomology ring of the symmetric product.  Presumably the tautological ring of $\Quot^e E$ admits a presentation similar to that of $\H^*(\Sym^n C)$ given in \cite{Mac}.\footnote{Note the slight error pointed out by Bertram and Thaddeus (\cite{BT}, 2.2).}  We leave the details to the reader.

\begin{rem} In case $E = \O_C^N$ is the trivial rank $N$ bundle, our $\Quot^{r,-d} \O_C^N$ is denoted $\Quot_d \O_C^N$ in \cite{Ber} and \cite{MO}.  If $P(X_1,\dots,X_r)$ is a polynomial in $r$ variables, our $$\int_{[\Quot^{r,-d} \O_C^N]^{\rm vir}} P(a_1,\dots,a_r) $$ is Bertram's $$N_d(P(X_1,\dots,X_r),g).$$  In this case, our Theorem~\ref{thm:virtualintersectionnumbers} is a special case of Proposition~2 in \cite{MO}.  Indeed, fix $s \in \{0,\dots,g \}$ and indices $1 \leq j_1 \leq \cdots \leq j_s \leq g$, so $b := b_{j_1}b_{g+j_1} \cdots b_{j_s}b_{g+j_s} \in \H^{2s}(\Quot \O_C^N)$.  Let $m := N(1-g+d)$ as usual.  Using their notation, let $R(x)=P(x) = x^{m-1+g-s}$, $J(x) = N/x$.  Let $\zeta_N$ be a primitive $N^{\rm th}$ root of unity.  A very special case of their formula yields \begin{eqnarray*} \int_{ [\Quot^{-d} \O_C^N]^{\rm vir} } a^{m-1+g-s} b &=& N^{-s} \sum_{i=1}^N (\zeta_N^i)^s R(\zeta_N^i) J(\zeta_N^i)^{g-1} \\ &=& N^{-s} \sum_{i=1}^N N^{g-1}(\zeta_N^i)^{N(1-g+d)} \\ &=& N^{g-s},  \end{eqnarray*} in agreement with our formula.  \end{rem}

\section{Curve Counting} \label{section:curvecounting}  We begin this section with a brief review of various curve counting theories, then in \S\ref{section:residueinvariants} we explain how to define the ``residue" invariants of a rank two bundle $E$ on a curve $C$ for each of these theories.  Finally, in \S\ref{section:correspondence} we recall the (somewhat conjectural) correspondence between these residue theories.

\subsection{DT theory} \label{section:DTtheory} For a smooth $3$-fold $X$ and $\beta \in \H_2(X)$, let $I_n(X,\beta)$ denote the Hilbert scheme of ideal sheaves $I \subset \O_X$ of $1$-dimensional subschemes $Z \subset X$ with $[Z] = \beta$ and $\chi(\O_Z) = n$.  Viewing this as the moduli space of rank one torsion free sheaves with trivial determinant (and the appropriate discrete invariants), the fixed determinant deformation theory of sheaves endows this space with a perfect obstruction theory (POT) and hence a virtual fundamental class $$[I_n(X,\beta)]^{\rm vir} \in A_e(I_n(X,\beta)).$$  The tangent and obstruction spaces at a point $I \in I_n(X,\beta)$ are given by the traceless $\Ext$ groups $$\Ext^1(I,I)_0, \; \Ext^2(I,I)_0,$$ respectively and the \emph{expected dimension} $e$ is given by $$\dim \Ext^1(I,I)_0 - \dim \Ext^2(I,I)_0 = \int_\beta c_1(TX)$$ (\cite{MNOP2}, Lemma~1).  The \emph{Donaldson-Thomas (DT)} invariants of $X$ are defined by pairing various cohomology classes on $I_n(X,\beta)$ with the virtual class.

\subsection{PT theory} \label{section:PTtheory} Let $P_n(X,\beta)$ denote the moduli space of (flat families of) pairs $(F,s)$ consisting of a sheaf $F$ on $X$ and a section $s \in \H^0(X,F)$ satisfying: \begin{enumerate} \item The support of $F$ is one dimensional and $[\Supp F]=\beta$. \item $\chi(F)=n$ \item $F$ is pure: any subsheaf $G \subset F$ with zero dimensional support is zero. \item $\Cok s$ has zero dimensional support. \end{enumerate}  The purity of $F$ ensures that $\Supp F$ is a Cohen-Macaulay (CM) curve of pure dimension one (no embedded points).  This is the \emph{stable pairs} moduli space of Pandharipande-Thomas.  It also carries a POT and a virtual class $$[P_n(X,\beta)]^{\rm vir} \in A_e(P_n(X,\beta))$$ reflecting the fixed determinant deformation theory of the two term complex $$I^\bullet := [ s: \O_X \to F ]$$ (sitting in degrees $0,1$) in the derived category of $X$; the deformation and obstruction spaces at a point $I^\bullet$ are again given by traceless $\Ext$ groups: $$\Ext^1(I^\bullet,I^\bullet)_0, \, \Ext^2(I^\bullet,I^\bullet)_0.$$  The expected dimension is again given by $$\dim \Ext^1(I^\bullet,I^\bullet)_0 - \dim \Ext^2(I^\bullet,I^\bullet)_0 = \int_\beta c_1(TX)$$ (the proof is the same as that of Lemma~1 in \cite{MNOP2} and is hence omitted in \cite{PT}).

It is a standard fact (\cite{PT}, 1.7) that the cokernel of $s$ is the ideal sheaf $I_Z$ of $Z := \Supp F$ in $X$, so we have an exact sequence $$0 \to I_Z \to \O_X \to F \to \Cok s \to 0,$$ hence $\chi(F) = \chi(\O_Z)+\chi( \Cok s)$.  Since $\Cok s$ has zero dimensional support, the second term is nonnegative and is just the length of $\Cok s$.  The \emph{Pandharipande-Thomas (PT)} invariants of $X$ are also defined by pairing various cohomology classes on $P_n(X,\beta)$ with the virtual class.

\subsection{DT=PT for minimal $n$} \label{section:DTequalsPT} Fix $\beta$ and consider the minimum $n$ so that $P_n(X,\beta)$ is non-empty.  Then we have an isomorphism $P_n(X,\beta) = I_n(X,\beta)$ identifying the POTs on the two spaces.  This is because minimality of $n$ ensures that every stable pair is just the natural surjection $\O_X \to \O_Z$ onto the structure sheaf of a CM curve $Z \subset X$ of pure dimension one.  Similarly, minimality of $n$ ensures that every curve $Z \in I_n(X,\beta)$ is CM of pure dimension one, else we could pass to the subcurve $Z' \subseteq Z$ defined by the largest ideal $I \subset \O_Z$ with zero dimensional support to obtain a curve with smaller Euler characteristic.  The structure sheaf of such a CM curve is pure, hence it defines a stable pair $\O_X \to \O_Z$.  The POTs are identified since the complex $I^\bullet = [\O_X \to \O_Z]$ associated to a stable pair of minimal Euler characteristic is quasi-isomorphic to the ideal sheaf of $Z$.

\subsection{Residue invariants} \label{section:residueinvariants} Let $E$ be a rank two vector bundle over a smooth proper curve $C$.  The total space of $E$ is a smooth $3$-fold with $\H_2(E) = \ZZ[C]$ generated by the class of the zero section; we will often just write $b$ for $b[C]$.  The GW, DT, and PT ``residue" invariants of $E$ are defined by formally applying the virtual localization formula using the natural $T$ action on $E$ by scaling (which induces a $T$-action on each moduli space in question making the perfect obstruction theory $T$-equivariant).  For example, the GW invariants of $E$ are defined by $$M_{g,b}(E) := \int_{ [\ov{M}^\bullet_{g,0}(C,b)]^{\rm vir} } e_T( - \pi_! f^* E) ,$$ where $\pi$ is the universal domain curve and $f$ is the universal map.  The integral is equivariant pushforward to a point and takes values in the localized equivariant cohomology ring of a point (i.e.\ the ring of rational functions of the generator $t$ of the equivariant cohomology ring of a point).  The superscript $\bullet$ indicates that the domain curve $\Sigma$ of a stable map may be disconnected (note $g := 1 - \chi(\O_\Sigma)$ can be negative) though no connected component can be contracted to a point.  Note that the $T$-fixed stable maps are just the stable maps to the zero section and the $T$-fixed part of the GW POT is nothing but the POT on stable maps to the zero section.

Similarly, the PT invariants of $E$ are defined by $$P_{n,b}(E) := \int_{ [P_n(E,b)^T]^{\rm vir} } e_T(- N^{\rm vir} ),$$ where the integral is again equivariant pushforward to a point.  Here $[P_n(E,b)^T]^{\rm vir}$ is the virtual class associated to the POT on $P_n(E,b)^T$ obtained from the $T$ fixed part of the POT on $P_n(E,b)$ and $N^{\rm vir}$ is the moving part of this POT (we will examine this closely in Section~\ref{section:obstructiontheory}), viewed as an element of (perfect, $T$-equivariant) $K$ ring $K_T(P_n(E,b)^T)$.  The DT invariants $I_{n,b}(E)$ are defined similarly.  It is not obvious that these invariants have any enumerative meaning, though our goal is to show that they do.

\subsection{GW/DT/PT residue correspondence} \label{section:correspondence}  The GW, DT, and PT residue invariants are all $T$-equivariant-deformation invariant, so they are independent of the choice of degree $d$ bundle $E$, hence we may write $M_{g,b}(d), I_{n,b}(d), P_{n,b}(d)$ in lieu of $M_{g,b}(E), I_{n,b}(E), P_{n,b}(E)$.

Let \begin{eqnarray*} Z^{\rm GW}_b(d)(u) & := & \sum_g M_{g,b}(d) u^{2g-2} \in \QQ(t)((u)) \\ Z^{\rm DT}_b(d)(q) & := & \sum_n I_{n,b}(d) q^n \in \ZZ(t)((q)) \\ Z^{\rm PT}_b(d)(q) & := & \sum_n P_{n,b}(d) q^n \in \ZZ(t)((q)) \end{eqnarray*} be the generating functions for these invariants.  Let $$M(q) := \prod_{n=1}^\infty \frac{1}{(1-q^n)^n}$$ denote the MacMahon function (the generating function for plane partitions).  The GW/DT correspondence proved in \cite{OP} says that \begin{eqnarray} \label{ZDT0} Z^{\rm DT}_0(d) &=& M(-q)^{8g-8-d}\end{eqnarray} (this is a special case of Conjecture~1 in \cite{BP2}), that the \emph{reduced DT partition function} $Z^{\rm red \; DT}_b(d) := Z^{\rm DT}_b(d) / Z^{\rm DT}_0(d)$ is a rational function of $q$, and that \begin{eqnarray} \label{GWDT} (-iu)^{b(2-2g+d)}Z_b^{\rm GW}(d) &=& (-q)^{(-b/2)(2-2g+d)} Z^{\rm red \; DT}_b(d).\end{eqnarray}  (These are the ``absolute" special cases of Theorems~1,2, and 3 in \cite{OP}.)  

The conjectural DT/PT correspondence of \cite{PT} asserts (for any $3$-fold $X$) that $Z^{\rm PT}_\beta(X)=Z^{\rm red \; DT}_\beta(X)$ for any $\beta \in \H_2(X)$.  In particular, this would imply \begin{eqnarray} \label{GWPT} (-iu)^{b(2-2g+d)}Z_b^{\rm GW}(d) &=& (-q)^{(-b/2)(2-2g+d)} Z^{\rm PT}_b(d).\end{eqnarray}   In fact, it is expected that the same TQFT method used to compute the GW \cite{BP1}, \cite{BP2} and DT \cite{OP} invariants of $E$ can also be used to compute the PT invariants of $E$, thus establishing \eqref{GWPT}.

Let us now focus on the case $b=2$, which we will study throughout.  The explicit formula \begin{eqnarray} \label{ZGW2} Z^{\rm GW}_2(d) & = & (ut)^{4g-4-2d}4^{g-1}(2 \sin \frac{u}{2})^{2d} \left ( (1+\sin \frac{u}{2})^{d+1-g} + (1- \sin \frac{u}{2})^{d+1-g} \right ) \\ \nonumber & = & (ut)^{4g-4-2d}2^{2g-2}(2 \sin \frac{u}{2})^{2d}  \sum_{i} \begin{pmatrix} d+1-g \\ 2i \end{pmatrix} 2 (\sin \frac{u}{2})^{2i} \\ \nonumber &=& (ut)^{4g-4-2d}  \sum_{i} \begin{pmatrix} d+1-g \\ 2i \end{pmatrix} 2^{2g-1-2i} (2\sin \frac{u}{2})^{2i+2d} \end{eqnarray}  can be found in Section~8 of \cite{BP2}.\footnote{It is a matter of definitions that, for integers $k_1,k_2$ satisfying $d=k_1+k_2$, the partition function ${\rm GW}_b(g|k_1,k_2)$ defined in \cite{BP2} is related to our $Z^{\rm GW}_b(d)$ by $$Z^{\rm GW}_b(d) = u^{b(2g-2-d)} {\rm GW}_b(g|k_1,k_2)_{t=t_1=t_2}    .$$}  (The binomial coefficient should be defined as in \cite{ACGH} so that the usual binomial expansion holds for negative exponents.)  When $b=2$, \eqref{GWPT} specializes to \begin{eqnarray} \label{GWPTdegree2} u^{4-4g+2d} Z^{\rm GW}_2(d) &=& q^{2g-2-d} Z^{\rm PT}_2(d). \end{eqnarray} Under the change of variables $ -q = e^{iu}$, we have $(2 \sin u/2)^2=q^{-1}(1+q)^2$, so by combining \eqref{ZGW2} with \eqref{GWPTdegree2} we arrive at the (conjectural) formula \begin{eqnarray} \label{ZPT} Z^{\rm PT}_2(d)  &=& t^{4g-4-2d} \sum_i \begin{pmatrix} d+1-g \\ 2i \end{pmatrix} 2^{2g-1-2i} q^{2-2g-i}(1+q)^{2d+2i} \end{eqnarray} for the residue PT invariants in degree $2[C]$.

\section{Maximal Subbundles} \label{section:maximalsubbundles}

Let $E$ be a rank $2$, degree $d$ vector bundle over a genus $g$ curve $C$.  Define integers $\epsilon \in \{ 0,1 \}$ and $e$ by the formula $$g-1 + \epsilon = d-2e.$$  Let $f$ be the largest integer such that a generic (rank $2$ degree $d$) stable bundle contains a line subbundle of degree $f$.  We wish to argue by dimension counting that $f=e$.  Recall that the dimension of the moduli space of rank $2$ stable bundles with fixed determinant is $3g-3$ (\cite{NS}, Theorem~1(iv)).  By definition of $f$, a generic stable bundle $E$ fits into a SES $$0 \to S \to E \to Q \to 0,$$ so the dimension of the moduli of stable bundles is bounded above by $$\dim \Jac C + \dim \PP \Ext^1(Q,S)$$ since $Q$ is determined by $S$ because $\det E$ is fixed.  On the other hand, stability of $E$ implies $\Hom(Q,S)=0$, so the dimension of the $\Ext$ group can be read off from Riemann-Roch and we obtain an inequality $$3g-3 \leq g - (1-g+2f-d) -1 $$ implying $f \leq e$. For the other inequality $e \leq f$ it is enough to show that $\Hom(S,E) \neq 0$ for some $S$ of degree $e$, which will follow from the construction (below) of the locus of such $S$ as a determinantal locus in $\Jac^e C$ together with the fact that the class of expected dimension $0$ supported on this locus is nonzero in homology.  By similar dimension counting arguments, one can show that the dimension of $\Quot^e E$ is $\epsilon$ for generic $E$, and that $\Quot^e E$ is generically smooth (so it is smooth when $\epsilon=0$).  

In case $\epsilon=0$ and $E$ is sufficiently generic, the smooth space $\Quot^e E$ is just a finite number of points; the fact that this number is $2^g$ was apparently known to Segre (in some form) in 1889 \cite{Seg}.  A ``modern" proof via Grothendieck-Riemann-Roch and the Thom-Porteous Formula can be given fairly easily.  We sketch the argument (following the proof of Theorem~3.1 in \cite{LN}) since it is relevant to our discussion of the geometry of the $\Quot$ scheme.\footnote{All of these statements make sense when $C$ is $\PP^1$, as long as one understands ``generic" to mean ``the splitting type is as balanced as possible".  By semicontinuity, this is an open condition in families.} 

Choose a smooth canonical divisor $j : D \hookrightarrow C$ (so $D$ is just $2g-2$ points of $C$) corresponding to a section $\O_C \to \omega_C$ vanishing on $D$.  Then we have an exact sequence \begin{eqnarray} \label{sequence1} 0 \to \O_C \to \omega_C \to j_* j^* \omega_C \to 0\end{eqnarray} on $C$. Let $$\iota = (\Id \times j): \Jac^e C \times D \to \Jac^e C \times C$$ be the inclusion, $\pi_1, \pi_2$ the projections from $\Jac^e C \times C$. Tensor \eqref{sequence1} with $E$, pull back to $\Jac^e C \times C$ and tensor with the dual of the universal degree $e$ line bundle $L_e$ to get an exact sequence $$0 \to L_e^\lor \otimes \pi_2^*E \to L_e^\lor \otimes \pi_2^* (E \otimes \omega_C) \to \iota_* \iota^* L_e^\lor \otimes \iota_* \iota^* \pi_2^* (E \otimes \omega_C) \to 0.$$  To save notation, set \begin{eqnarray*} F & := & L_e^\lor \otimes \pi_2^* (E \otimes \omega_C) \\ G & := & \iota_* \iota^* L_e^\lor \otimes \iota_* \iota^* \pi_2^* (E \otimes \omega_C).\end{eqnarray*} For a sufficiently generic $E$ one can argue that $R^1 \pi_{1*} F=0$ and that the number of maximal subbundles is the degeneracy locus of the map $\pi_{1*}F \to \pi_{1*} G$ of vector bundles on $\Jac^e C$.  (We will see in a moment that the difference in ranks here is $$\ch_0(\pi_{1*}F-\pi_{1*}G) = g - 1 = \dim \Jac^e C-1,$$ so one expects this degeneracy locus to be zero dimensional.)  By the Thom-Porteous formula, if this degeneracy locus is smooth and zero dimensional (which it will be for generic $E$), then the number of points in it is $$\int_{\Jac^e C} c_g( \pi_{1*} G - \pi_{1*} F ).$$  

We adopt the notation of Section~\ref{section:quotschemes}.   The formula \eqref{c1ofL} for $c_1(L_e)$ implies that it is topologically trivial on $\Jac^e C \times D$.  Since $\iota^* L_e^\lor \otimes \iota^* \pi_2^* (E \otimes \omega_C)$ differs from $\iota^*L_e^\lor$ only by tensoring with a bundle pulled back from $D$ (hence trivial), it is also topologically trivial, hence so is $\pi_{1*} G$ and we have $\ch \pi_{1*}G = 4g-4.$  Using formula \eqref{chL}, we have \begin{eqnarray*} \ch F & = & (\ch L_e^\lor )( \pi_2^* \ch E )(\pi_2^* \ch \omega_C) \\ & = & (1-e \eta-\gamma - \eta \theta)(2+d\eta)(1+(2g-2)\eta) \\ & = & 2-2\gamma+(5g-5)\eta-2 \eta \theta, \end{eqnarray*} so by GRR we compute \begin{eqnarray*} \ch \pi_{1*}F & = & \ch \pi_{1!} F \\ & = & \pi_{1*}( (\ch F )(\td \pi_1)) \\ &=& \pi_{1*} ( (2-2\gamma+(5g-5)\eta-2 \eta \theta)(1+(1-g)\eta) )\\ & = & 5g-5-2 \theta. \end{eqnarray*}  Since $\ch(\pi_{1*}G-\pi_{1*}F) = 1-g + 2 \theta$, it follows from an exercise with symmetric functions (c.f.\ page 336 in \cite{ACGH}) that $c(\pi_{1*}G-\pi_{1*}F) = \exp (2 \theta)$, so finally we compute $$\int_{\Jac^e C} c_g( \pi_{1*}G- \pi_{1*}F) = \int_{\Jac^e C} \frac{(2 \theta)^g}{g!} = 2^g$$ using \eqref{thetatotheg}.

Let us see how to intepret the $2^g$ maximal subbundle count as a DT=PT invariant.  If we expand \eqref{ZPT} as a Laurent series in $q$, the smallest power of $q$ with nonzero coefficient occurs when $i$ takes the value $e$ determined by the equation $$g-1+\epsilon = d-2e$$ ($\epsilon \in \{0,1 \}$ as usual).  When $\epsilon=0$ (i.e.\ when $d$ and $g$ have opposite parity), this lowest order term is $$t^{4g-4-2d} 2^{3g-2-d} q^{2-2g-e}$$ so the DT=PT invariant in minimal Euler characteristic is given by \begin{eqnarray} \label{PT} P_{2-2g-e,2}(d) & = &   t^{4g-4-2d} 2^{3g-2-d}. \end{eqnarray}

The equality of DT and PT invariants in minimal Euler characteristic as discussed in \ref{section:DTequalsPT}, together with the known equivalence of DT and GW mentioned above, ensures that this is actually the correct PT invariant even though the coefficients of higher powers of $q$ are technically only conjectural.  

Assume we are in the $\epsilon=0$ case so $g-1 = d-2e$.  Let $Y$ be the space ($\Quot$ scheme) of such minimal (i.e.\ degree $e$) line subbundles of $E$.  As mentioned above, if $E$ is sufficiently generic, $Y$ is just a finite number of points.  On the other hand, we can identify $Y$ with the $T$-fixed subscheme of $$P := I_{2g-2-e}(E,2) = P_{2g-2-e}(E,2)$$ using Theorem~\ref{thm:Tfixedstablepairs} below (or by using Proposition~\ref{prop:CMcurvesinE} and the general remarks about this common moduli space in minimal Euler characteristic).  Since $Y$ is smooth and zero dimensional, the tangent space to $Y$ at a point $L \hookrightarrow E$ is zero.  On the other hand, this tangent space is given by $\Hom(L, Q) = \H^0(C,L^\lor Q),$ where $Q = E/L$ as usual.  The degree of $L^\lor Q$ is $-e+d-e = g-1$, so by Riemann-Roch, we also have $\H^1(C,L^\lor Q)=0$.  Let $Z = \Spec_C \O_C[L^\lor]$ be the degree $2[C]$ curve in $E$ corresponding to a point $L \in Y$.  The DT=PT deformation and obstruction spaces at this point are given by $$\H^0(Z,N_{Z/E}), \, \H^1(Z,N_{Z/E})$$ respectively.  This is a consequence of the results of Section~\ref{section:obstructiontheory}, as mentioned in Remark~\ref{rem:simplifiedvnb}.  In fact, the statement about the deformation space can be seen directly, since the moduli space in question is a Hilbert scheme.  The identification of the obstruction space can also be derived from Proposition~4.4 in \cite{PT} (or its proof).  The point is that the stable pairs in question are of the ideal form: they are structure sheaves of embedded CM curves.

We will see in the next section (Equation~\ref{normalbundleofZ}) that the normal bundle fits into a SES $$0 \to \pi^* L^{\otimes 2} \to N_{Z/E} \to \pi^*Q \to 0.$$ Since $\pi$ is an affine morphism, we can compute global sections (and the higher direct images of pushforward to a point) by first pushing forward to $C$ via $\pi$.  Note $\pi_* \pi^* G = G \oplus G L^\lor$ for any sheaf $G$ on $C$, so pushing forward this exact sequence to $C$ we have a SES $$0 \to L_t \oplus L_{2t}^2 \to \pi_* N_{Z/E} \to Q_t \oplus (QL^\lor)_0 \to 0$$ on $C$, where the subscripts indicate the weight of the natural $T$ action.  We've already observed that $$\H^0(C,QL^\lor)=\H^1(C,QL^\lor)=0,$$ so the fixed part of the $T$-equivariant POT on $Y$ is trivial, and the virtual fundamental class on $Y$ is its usual fundamental class.  We conclude that the entire POT is moving, so the virtual normal bundle (at a point $(L \hookrightarrow E) \in Y$) is then \begin{eqnarray*} N^{\rm vir}_{ \{ L \} / P} &=& \H^0(Z,N_{Z/E})-\H^1(Z,N_{Z/E}) \\ & = & \H^0(C,Q)_t + \H^0(C,L)_t + \H^0(C,L^2)_{2t} \\ & & -\H^1(C,Q)_t - \H^1(C,L)_t - \H^1(C,L^2)_{2t} \end{eqnarray*} (viewed as an element of the $T$ equivariant $K$ group of a point).  By Riemann-Roch we have \begin{eqnarray*} \chi(Q) & = & 1-g+d-i \\ \chi(L) & = & 1-g+e \\ \chi(L^2) & = & 1-g+2e, \end{eqnarray*} so we compute $$e_T(-N^{\rm vir}_{ \{ L \} / P})= t^{g-1+e-d}t^{g-1-e}(2t)^{g-1-2e} = 2^{2g-2-d}t^{3g-3-d-2e}.$$ Since the DT=PT invariant is given by \begin{eqnarray*} P_{2g-2-e,2}(E) & = & \sum_{L \in Y} e_T ( -N^{\rm vir}_{ \{ L \} / P } ) \\ &=& (\# Y) 2^{2g-2-d}t^{3g-3-d-2e} \\ & = & (\# Y) 2^{2g-2-d}t^{4g-4-2d},\end{eqnarray*} we see that $\#Y = 2^g$ by using the known value of this DT=PT invariant given in \eqref{PT}.

Notice that $P_{2g-2-e}(E,2)$ actually \emph{counts} the number of $T$-invariant CM curves in class $2[C]$ in a generic $E$.  To my knowledge, there is no such reasonable enumerative interpretation of the residue Gromov-Witten invariants.

Of course, if one wishes to compute this minimal Euler characteristic PT invariant without using deformation invariance and the choice of a generic $E$, then one simply replaces the $2^g$ maximal subbundle \emph{count} with the virtual class formula of Corollary~\ref{corollary:zerodimensionalvirtualclass} and the general formula \eqref{eTNvir} for $e_T(-N^{\rm vir})$ which we will derive later.

\section{$T$ fixed curves and pairs}

The main goal of this section is to identify the CM curves and stable pairs on $E$ fixed by the scaling action.  As mentioned in the Introduction, we will only treat the case of homology class $2[C] \in \H_2(E,\ZZ)$ here, leaving the general case for \cite{Gil2}.  Though it is not used elsewhere, we have also included a description of the fixed locus of the zero dimensional Hilbert scheme of the total space of a line bundle $L$ on $C$ (Theorem~\ref{thm:TfixedHilbertscheme}).

\begin{prop} \label{prop:CMcurvesinE} Let $\pi : E \to C$ be a vector bundle over a smooth curve.  Suppose $Z \subseteq E$ is a proper CM curve in $E$ invariant under the scaling action and in homology (or Chow) class $2[C]$.  Then $Z$ is the doubling of the zero section along a line  subbundle $L \subset E$. \end{prop}

\begin{proof}  Certainly such a $Z$ must be supported on the zero section, so $\O_Z = \pi_* \O_Z$ is a sheaf of $\O_C$-algebras.  Since $Z$ is $T$-invariant, $\O_Z$ has a grading making $\Sym^* E^\lor \to \O_Z$ a surjection of graded $\O_C$-algebras.  Since $Z$ is CM, each graded piece of $\O_Z$ must be a locally free $\O_C$-module, else its torsion submodule violates purity.  Since $\O_Z$ has length $2$ at the generic point of $C$, $\O_Z$ is nonzero only in grading zero (where it must be $\O_C$ since $\O_C = \Sym^0 E^\lor$ surjects onto it) and one other grading.  Since $\Sym^* E^\lor$ is generated in grading $1$, this other grading must be $1$ and we have a surjection $E^\lor \to (\O_Z)_1$ in grading $1$.  The locally free sheaf sheaf $L := (\O_Z)_1^\lor \subset E$ must be a line bundle because of the homology class requirement.  Evidently $\O_Z = \O_C[L^\lor]$ is the trivial square zero extension of $\O_C$ by $L^\lor$. \end{proof}

\begin{rem} Note that the $T$-fixed subscheme of such a $Z$ is nothing but the zero section $C$.  The scheme $Z$ is l.c.i. and hence $Z$ is a Cartier divisor when $E$ has rank $2$.  The map $\pi : Z \to C$ is flat and finite of degree $2$.  We can compute $\chi(\O_Z)$ by first pushing forward to $C$: $$\chi(\O_Z) = \chi(\O_C \oplus L^\lor) = 2g-2-\deg L.$$ Since $Z$ is the first infinitesimal neighborhood of $C$ in $L$, we also have a closed embedding $\iota : C \to Z$ corresponding to the surjection $\O_C[L^\lor] \to \O_C$ with kernel $L^\lor$.  Evidently $\pi \iota = \Id_C$.  We find that \begin{eqnarray*} \iota^* \pi^* &=& \Id \\ \pi_* \iota_* &=& \Id \\ \pi_* \pi^* G &=& G \oplus L^\lor G. \end{eqnarray*} \end{rem}

We now assume $E$ has rank $2$.

\begin{rem} \label{rem:otherTfixedcurves} One can show similarly that such a $Z$ in class $3[C]$ is either $\Spec \O_C[E^\lor]$ (the first infinitesimal neighborhood of the zero section in $E$) or the second infinitesimal neighborhood $$\Spec_C ( \O_C \oplus L^\lor \oplus L^{\lor \otimes 2} )$$ of the zero section in a line subbundle $L \subset E$.  The latter is l.c.i.\ in $E$ while the former is not even Gorenstein.  Which of the two has smaller Euler characteristic depends on the relationship between $d$ and $g$.  One can again describe and filter the normal bundle as in the above proof.  \end{rem}

Let $Z := \Spec \O_C[L^\lor]$ be the doubling of $C$ along $L \hookrightarrow E$ as in the proposition, and let $I$ be the ideal sheaf of $Z$ in $E$ and $Q:=E/L$ the locally free quotient.  Then we have a short exact sequence \begin{eqnarray} \label{normalbundle} 0 \to \pi^* Q^\lor \to I/I^2 \to \pi^* L^{\lor 2} \to 0 \end{eqnarray} of sheaves on $Z$.  To see this, note that $I$ is defined by the SES $$0 \to I \to \Sym^* E^\lor \to \O_Z \to 0,$$ which we can analyse grading by grading.  In grading zero the second map is an isomorphism, so $I_0 = 0$, and in grading $1$ it is $E^\lor \to L^\lor$, so $I_1 = Q^\lor$.  In higher gradings the second map is zero, so $I_n = \Sym^n E^\lor$ when $n \geq 2$.  Now we analyse the inclusion $I^2$ and the inclusion $I^2 \subset I$ in each grading; we will write $I^2_n$ for the degree $n$ part $(I^2)_n$ of $I^2$.  We have $I^2_1 = 0$, so $(I/I^2)_1 = Q^\lor$.  We have $I^2_2 = \Sym^2 Q^\lor$, so $$(I/I^2)_2 = \Sym^2 E^\lor / \Sym^2 Q^\lor = E^\lor \otimes L^\lor$$ by linear algebra.  In grading $3$, $I^2_3 \hookrightarrow I_3$ is $$ Q^\lor \otimes \Sym^2 E^\lor \hookrightarrow \Sym^3 E^\lor,$$ and the quotient $(I/I^2)_3$ is $L^{\lor 3}$ by linear algebra.  In grading $n>3$ we have $I_n = I^2_n$.  Thus, as a graded $\O_C$ module we have $$\pi_*(I/I^2) = Q^\lor \oplus L^\lor E^\lor  \oplus L^{\lor 3},$$ and this has the ``obvious" structure of a locally free $\O_C[L^\lor]$-module of rank $2$ (the direct sum decompositions here and elsewhere are only as $\O_C$ modules, not as $\O_Z$ modules).  We have $\pi^* G = G \oplus L^\lor G$ for any $\O_C$-module $G$, so $\pi^* Q^\lor = Q^\lor \oplus L^\lor Q^\lor$ and the natural injection $$Q^\lor \oplus L^\lor Q^\lor \to Q^\lor \oplus L^\lor E^\lor  \oplus L^{\lor  3}$$ of $\O_C$-modules is in fact easily seen to be an $\O_Z$-module map with quotient $\pi^* L^{\lor 2} = L^{\lor 2} \oplus L^{\lor 3}$.  In fact, the dual \begin{eqnarray} \label{normalbundleofZ} 0 \to \pi^* L^2 \to N_{Z/L} \to \pi^*Q \to 0 \end{eqnarray} of the exact sequence \eqref{normalbundle} is nothing but the normal bundle sequence $$0 \to N_{Z/L} \to N_{Z/E} \to N_{L/E}|_Z \to 0.$$  

\begin{thm} \label{thm:TfixedHilbertscheme} Let $\pi : L \to C$ be a line bundle over smooth curve endowed with the scaling action.  Then there is an isomorphism of schemes $$\coprod_{\lambda} C^{l(\lambda)} / \Aut \lambda \to (\Hilb^n L)^T,$$ where the coproduct is over partitions $\lambda$ of $n$, $l(\lambda)$ denotes the length of $\lambda$, and $\Aut(\lambda)$ is the automorphism of group of $\lambda$.  \end{thm}

\begin{rem} It is interesting that $(\Hilb^n L)^T$ does not depend on $L$, though, for example, its normal bundle in $\Hilb^n L$ does depend on $L$.  \end{rem}

\begin{proof} We begin by constructing a map from the LHS to the RHS.  Let $\lambda = (1)^{m_1} \cdots (k)^{m_k}$ be a typical partition of $n$, so $$1m_1+2m_2+\cdots +km_k=n.$$  Let \begin{eqnarray*} X & := & C^{l(\lambda)} / \Aut \lambda \\ & = & \Sym^{m_1} C \times \cdots \times \Sym^{m_k} C  \end{eqnarray*} and let $\pi_i : X \to \Sym^{m_i} C$ be the projection.  Let $Z_i \subset \Sym^{m_i} C \times C$ be the universal Cartier divisor, let $W_i := (\pi_i \times \pi)^*Z_i$ be its pullback to $X \times L$, and let $I_i \subset \O_{X \times L}$ be the ideal of $W_i$ in $X \times L$.  Let $$J := \Sym^{\geq 1} \pi_{C}^* L^\lor$$ be the ideal of $X \times C$ in $X \times L$.  Consider the ideal $$K := I_1 \cdots I_k + I_2 \cdots I_k J + I_3 \cdots I_k J^2 + \cdots + I_k J^{k-1} + J^k \subset \O_{X \times L}.$$  Notice that $K$ is generated by ``monomials in $J$" so the subscheme $Z \subset X \times L$ that it defines is certainly $T$ invariant.  

I claim that $Z$ has length $n$ at every point of $X$.  Since $J^k \subset K$, certainly $Z$ is topologically supported on $X \times C \subset X \times L$; in fact it is scheme theoretically supported on the $(k-1)$-st infinitesimal neighborhood of $X \times C$ in $X \times L$ and is relatively zero dimensional over $X$.  Let $D=(D_1,\dots,D_k)$ be a point of $X$ and let $$Z_D \subset \{ D \} \times L = L$$ be the zero dimensional subscheme of $L$ determined by $Z$ at $D$.  For a point $P \in C$ and a zero dimensional $W \subset C$, write $\ell(W,P)$ for the length of $W$ at $P$.  Then by definition of $X$ we have $$\sum_{P \in C} \ell(D_i,P) = m_i$$ and we want to prove $$\sum_{P \in C} \ell(Z_D,P) = m_1+2m_2+\cdots+km_k=n,$$ so it suffices to prove $$\ell(Z_D,P) = \ell(D_1,P)+2 \ell(D_2,P)+\cdots + k \ell(D_k,P)$$ for each $P \in C$ for then we get the desired equality by summing over $P$.  Set $\ell_i := \ell(D_i,P)$ to ease notation.  Let $x$ be a local coordinate on $C$ centered at $P$ (a generator of $\m_P \subset \O_{C,P}$ if you like) and let $y$ be a local coordinate in the $L$ direction corresponding to a trivialization of $L$ near $P$.  Then, at $P$, the ideal $I_i$ is given by $(x^{\ell_i})$ and the ideal $J$ is given by $(y)$, so at $P$, the ideal $K$ is given by $$(x^{\ell_1+\cdots+\ell_k}, x^{\ell_2+\cdots+\ell_k}y, \dots , x^{\ell_k} y^{k-1}, y^k),$$ and it is easy to see, by drawing a picture of the corresponding Young diagram, that this monomial ideal has length $\ell_1+2\ell_2+\cdots+k \ell_k$ at the origin $P$.

We have proved that $Z$ is a flat family of $T$-invariant length $n$ subschemes of $L$ parameterized by $X$, hence we obtain a map $X \to (\Hilb^n L)^T.$  Taking the coproduct of these maps over all partitions gives the purported isomorphism.  To see that this is in fact an isomorphism, recall that $\Hilb^n L$ is smooth, hence so is the $T$-fixed locus $(\Hilb^n L)^T$, so by general nonsense we may conclude that our map is an isomorphism of schemes simply by checking that it is bijective on points.  Suppose $Z \subset L$ is a zero dimensional $T$ invariant subscheme of $L$ corresponding to a point of $(\Hilb^n L)^T$.  Then at each point $P$ if we choose local coordinates $x,y$ on $L$ near $P$ as above, then the ideal $I$ of $Z$ at $P$ can be generated by monomials in $y$ since $Z$ is $T$ invariant.  The coefficient $f(x)$ of a monomial $y^\ell$ is an arbitrary polynomial in $x$, but if we factor out the largest power of $x$ dividing $f(x)$, then what is left over is a unit near $P$.  We find that $I$ is a monomial ideal at $P$, so we can write it uniquely as $$(x^{\ell_1(P)+\cdots+\ell_k(P)}, x^{\ell_2(P)+\cdots+\ell_k(P)}y, \dots , x^{\ell_k(P)} y^{k-1}, y^k),$$ for some nonnegative integers $\ell_1(P),\dots,\ell_k(P)$ (independent of the choice of coordinates $x,y$).  Set $m_i := \sum_{P \in C} \ell_i(P)$.  Then the fact that $Z$ has length $n$ easily implies $$m_1+2m_2+\cdots+km_k = n$$ and it is easy to see that $Z$ is in the image of the component on the LHS corresponding to this partition of $n$.  \end{proof}

\begin{thm} \label{thm:Tfixedstablepairs} There is an isomorphism of schemes $$\coprod_{2n-e=m} \Quot^e E \times \Sym^n C \cong P_{2-2g+m}(E,2)^T.$$ \end{thm}

\begin{proof} First we define a map from the LHS to the RHS.  Let $Y := \Quot^e E \times \Sym^n C$ to ease notation.  Let $$0 \to S \to \pi_2^* E \to Q \to 0$$ be the universal sequence on $Y \times C$ pulled back from $\Quot^e E \times C$ and let $$s_D : \O_{Y \times E} \to \O_{Y \times E}(D)$$ be the pullback of the universal section of the universal Cartier divisor on $\Sym^n C \times C$ via the projection $Y \times E \to \Sym^n C \times C$.  Let $$t : \Sym^*_{\O_{Y \times C}} \pi_2^* E^\lor \to \O_{Y \times C}[S^\lor]$$ be the induced map on symmetric algebras obtained by quotienting $\Sym^* S^\lor$ by the ideal $\Sym^{\geq 2} S^\lor$, viewed simply as a map of graded $\Sym^* \pi_2^* E^\lor$ modules.  In degree $1$, $$t_1 : \pi_2^* E^\lor \to S^\lor$$ may not be surjective since $Q$ may not be locally free.  However, the sheaf $F := \O_{Y \times C}[S^{\lor}]$ is certainly a pure sheaf, topologically supported on $C$ at each point of $Y$.  Composing $$t : \O_{Y \times E} \to F$$ with the natural map $1 \otimes s_D : F \to F(D)$ defines a map $$ s : \O_{Y \times E} \to F(D), $$ which is evidently a family of stable pairs over $Y$ with the correct discrete invariants.  Furthermore, $s$ is clearly a graded map of graded sheaves (with $F(D)_0 = \O_C(D)$ and $F(D)_1 = S^\lor(D)$), hence it defines a $T$ invariant stable pair, that is, a morphism $$Y \to P_{2-2g+m}(E,2)^T.$$  

Now we define a map from the RHS to the LHS.  Suppose $s : \O_E \to F$ is in $P := P_{2-2g+n}(E,2)^T$.  Then, pushing forward to $C$, we may view $s$ as a graded map of graded $\Sym^* E^\lor$ modules $$s : \Sym^* E^\lor \to F.$$  By purity of $F$, each graded piece $F_d$ must be a locally free $\O_C$-module.  Since $s$ has zero dimensional cokernel, $s_0 : \O_C \to F_0$ must be a map of locally free sheaves on $C$ with zero dimensional cokernel, so it must be $s_D : \O_C \to \O_C(D)$ for some Cartier divisor $D \subset C$.  Since $F_0$ varies flatly over the universal family, so does $D$, and we obtain a map $P \to \Sym^d C$ for some $d$.  In degree $1$, $s$ is a map $s_1 : E^\lor \to F_1$.  The condition $[\Supp F]=2[C]$ and the fact that $\Sym^* E^\lor$ is generated in degree $1$ ensure that the locally free sheaf $F_1$ has rank $1$, and $F_i=0$ for $i>1$.  The fact that $s$ is a map of graded $\Sym^* E^\lor$ modules means we have a commutative diagram $$\xymatrix@C+10pt{ E^\lor \otimes \O_C \ar[d]_m^{\cong} \ar[r]^{1 \otimes s_0} & E^\lor \otimes F_0 \ar[d]^a \\ E^\lor \ar[r]^{s_1} & F_1 }$$ expressing the compatibility (with $s$) of the action $a$ of $\Sym^* E^\lor$ on $F$ and the action of $\Sym^* E^\lor$ on itself by multiplication.  Using our identification of $s_0 = s_D$, we may view this diagram as providing a factorization of $s_1$ through a map $t : E^\lor \to F_1(-D)$.  Since $\Cok t$ has zero dimensional support, we have $\sHom(\Cok t, \O_C)=0$, so $t^\lor : F_1^\lor(D) \to E$ is monic.  Noting that $F_1^\lor(D)$ varies flatly over $P$, the dual of $t$ therefore yields a map $P \to \Quot^e E$ for the unique $e$ satisfying $2n-e=m$.  Taking the product of this map and the previous map we get a map $P \to Y$.

Since the two maps are easily seen to be inverse, the isomorphism is established. \end{proof}

\begin{rem} \label{rem:Fasacurve} Observe that the sheaf $F = \O_C[S^\lor]$ in a stable pair as in the above theorem is itself the structure sheaf of a (Cohen-Macaulay) curve with the same topological space as $C$ (namely the first infinitesimal neighborhood of the zero section $C$ in $S$).  The section $s$ may be viewed as a (finite) map of schemes $s : \Spec_C F \to E$ (over $C$) which factors through the inclusion $$Z = \Spec_C \O_C[L^\lor] \hookrightarrow E$$ and which is generically an isomorphism onto $Z$.  Locally, where $S$ and $L$ can be trivialized, $s : \Spec_C \to Z$ looks like $\Spec_C$ of the map of $\O_C$-algebras \begin{eqnarray*} \O_C[y]/y^2 & \to & \O_C[z]/z^2 \\ y & \mapsto & fz \end{eqnarray*} for some $f \in \O_C$.

In particular, $(F,s)$ defines a point of Honsen's moduli space of finite maps from CM curves to $E$, which is in keeping with the original motivation for considering the stable pairs space.  It might be interesting to compare the deformation theory of the stable pair $(F,s)$ with the deformation theory of the map $s : \Spec_C F \to E$, possibly viewing the latter in some moduli space where the domain of the map is also allowed to vary.  \end{rem}

\section{Obstruction Theory}  \label{section:obstructiontheory}

In this section, we will identify the POT on $$Y := \Quot^e E \times \Sym^n C$$ inherited from the fixed part of the stable pairs POT via the isomorphism of Theorem~\ref{thm:Tfixedstablepairs} as well as the virtual normal bundle of $Y$ in the stable pairs space.  It will be helpful if the reader is familiar with virtual localization \cite{GP}.  A more general study of this POT can be found in \cite{Gil2}.

In general, if $P = P_n(X,\beta)$ is a stable pairs space and $\II^\bullet$ is the universal stable pair, then an important fact (\cite{PT}, 2.2) subtly alluded to above is that $P$ can also be viewed as a moduli space of derived category objects.  In particular, if $j: U \to P$ is determined by a stable pair $I^\bullet = [\O_{U \times X} \to F]$ parameterized by $U$, then $I^\bullet = \bL (j \times \Id_X)^* \II^\bullet$ in $D(U \times X)$.  We will use this when $j$ is the inclusion $Y \hookrightarrow P$ of a component of the $T$ fixed locus.  The POT on $P$ is given by a morphism $$\bR \pi_* \bR \sHom(\II^\bullet, \II^\bullet)_0 \otimes \omega_\pi [2] \to \LL_P$$ in $D(P)$ to the cotangent complex of $P$ defined using the Atiyah class of $\II^\bullet$ (see \cite{PT}, 2.3).

The virtual normal bundle of $Y$ in $P$ is, by definition, the image of the moving part of \begin{eqnarray} \label{virtualnormalbundle} \bL j^* \bR \sHom( \bR \pi_* \bR \sHom(\II^\bullet, \II^\bullet)_0 \otimes \omega_\pi [2]   , \O_P) \end{eqnarray} under the map $D_T(Y) \to K_T(Y)$ from the (perfect $T$ equivariant) derived category of $Y$ to its Grothendieck ring.  The $T$ fixed part of \eqref{virtualnormalbundle} defines a POT on $Y$ (using the isomorphism $(\bL j^* \LL_P)^T \to \LL_Y$) which we will also have to identify.  Using Serre duality for $\pi$, cohomology and base change for the square $$\xymatrix@C+15pt{ Y \times E \ar[r]^{j \times \Id_E} \ar[d]_\pi & P \times E \ar[d]^\pi \\ Y \ar[r]^j & P}$$ and the (obvious) ``self duality" $$\bR \sHom( \II^\bullet, \II^\bullet)_0 = \bR \sHom( \bR \sHom( \II^\bullet, \II^\bullet)_0 , \O_{P \times E}),$$ we can rewrite \eqref{virtualnormalbundle}: \begin{eqnarray*} & & \bL j^* \bR \sHom( \bR \pi_* \bR \sHom(\II^\bullet, \II^\bullet)_0 \otimes \omega_\pi [2]   , \O_P) \\ & = & \bL j^* \bR \pi_* \bR \sHom( \bR \sHom(\II^\bullet,\II^\bullet)_0 \otimes \omega_\pi [2], \omega_\pi[3]) \\ & = & \bR \pi_* \bL (j \times \Id_E)^* \bR \sHom(\II^\bullet, \II^\bullet)[1] \\ & = & \bR \pi_* \bR \sHom( \bL (j \times \Id_E)^* \II^\bullet, \bL (j \times \Id_E)^* \II^\bullet)_0[1] \\ & = & \bR \pi_* \bR \sHom(I^\bullet,I^\bullet)_0[1].\end{eqnarray*} 

\subsection{Notation} Let $p : Y \times C \to Y$ and $\pi : Y \times E \to Y$ be the projections, $\iota : Y \times C \to Y \times E$ the zero section, so $p=\pi \iota$.  We drop all notation for pullbacks, so we will just write $$ 0 \to S \to E \to Q \to 0$$ instead of $$0 \to \pi_{13}^* S \to \pi_3^* E \to \pi_{13}^* Q \to 0$$ for the pullback to $Y \times C$ of the universal sequence on $\Quot^e E \times C$.  Similarly, we write $D \subset Y \times C$ instead of $\pi_{23}^* D$ for the pullback of the universal Cartier divisor, we write $\O$ for $\O_{Y \times C}$ and sometimes we just write $D$ for the invertible sheaf $\O(D)$.  Let $F'=\O[S^\lor]$, $F = D[S^\lor D]$, and let $I^\bullet = [ \O \to F]$ be the universal stable pair on $Y \times E$ as described in the proof of \eqref{thm:Tfixedstablepairs}.  The pair $I^\bullet$ corresponds to the map $j : Y \to P$ giving an isomorphism onto a component of the $T$ fixed subscheme.

\begin{lem} \label{lem:virtualnormalbundle} The virtual normal bundle of $Y$ in $P$ is given by \begin{eqnarray} \label{Nvir} N^{\rm vir}_{Y/P} &=& p_!(S^\lor D + D^\lor \land^2 E + D^\lor S \land^2 E + 2E + SE) \\ \nonumber & & \quad -p_!( S^\lor+2 \land^2 E +S \land^2 E + S + S^\lor \land^2 E)   \quad \in K_T(Y).  \end{eqnarray}  \end{lem}

\begin{proof} This is easy to prove because we can calculate everything on the level of the ($T$ equivariant) Grothendieck rings $K_T=K_T(Y \times E)$, $K_T(Y \times C)$, and $K_T(Y)$.  We just need the classes of various derived duals---these are easier to calculate than actual duals or $\Ext$ groups.  We seek a formula for the moving part of \begin{eqnarray*} -\pi_! \bR \sHom(I^\bullet,I^\bullet)_0 &=& \pi_! \O_{Y \times E} - \pi_! \bR \sHom(I^\bullet, I^\bullet) \end{eqnarray*} in $K_T(Y)$ (the minus sign results from the shift $[1]$).  The filtration of $F = D[S^\lor D]$ by gradings $$0 \to F^{\geq 1} \to F \to F / F^{\geq 1} \to 0$$ is identified with $$0 \to \iota_* S^\lor D \to F \to \iota_* D \to 0,$$ so we have $F = \iota_* D + \iota_* S^\lor D$ in $K_T(Y \times E)$.  The Koszul complex \begin{eqnarray} \label{Koszulresn} \land^2 E^\lor \to E^\lor \to \O_{Y \times E} \to \iota_* \O_{Y \times C} \end{eqnarray} is a ($T$-equivariant for the obvious actions) resolution of $\iota_* $ by vector bundles on $Y \times E$.  Similarly, if $G$ is any vector bundle on $Y \times C$, we can resolve $\iota_* G$ by tensoring \eqref{Koszulresn} with $\pi^* G$, so we have $$\iota_* G = G( \land^\bullet E^\lor)$$ in $K_T$.  Here we have denoted multiplication in $K_T$ by juxtaposition (so we often just write $1$ for a structure sheaf) and we use the shorthand $\land^\bullet E := \sum_i (-1)^i \land^i E$.  (We are just reproducing the proof of GRR for $\iota$.)  

For the rest of the proof we use $F^\lor, (F')^\lor$ for the classes of the derived duals in $K_T$.  Putting the Koszul resolution together with the filtration of $F$, we find that \begin{eqnarray*} F & = & \iota_* D + \iota_* S^\lor D \\ &=& ( D + S^\lor D)(\land^\bullet E^\lor) \\ F^\lor & = & (D^\lor \land^2 E + D^\lor S \land^2 E)(\land^\bullet E^\lor) \\ &=& \iota_*(D^\lor \land^2 E + D^\lor S \land^2 E) \\ F F^\lor & = & (F')(F')^\lor \\ & = & (1 + S)( 1 + S^\lor)(\land^\bullet E)(\land^\bullet E^\lor)) \\ & = & \iota_*(2-2E+2\land^2 E + S -ES + S \land^2 E + S^\lor - S^\lor E + S^\lor \land^2 E) \end{eqnarray*} in $K_T$.  Here we have used the ``Koszul duality" identity $(\land^\bullet E )(\land^2 E^\lor) = \land^\bullet E^\lor$, which is basic linear algebra.  Since $I^\bullet = \O_{Y \times E} - F$ in $K_T$, we have \begin{eqnarray*} \bR \sHom(I^\bullet,I^\bullet) & = & (\O_{Y \times E}-F)(\O_{Y \times E}-F^\lor) \\ & = & \O_{Y \times E} + \iota_*(2+2\land^2 E + S + S \land^2 E + S^\lor + S^\lor \land^2 E) \\ & & - \iota_*(D+S^\lor D + D^\lor \land^2 E + D^\lor S \land^2 E + 2E + SE) \end{eqnarray*} in $K_T(Y \times E)$.  We can subtract off the first term to get the traceless part.  The $T$-fixed part of this expression is $$\iota_*(2-D-S^\lor E ) ,$$ which we can discard to get the moving part.  The result now follows because $\pi \iota = p,$ so $\pi_! \iota_* = p_!.$  \end{proof}

\begin{rem} \label{rem:simplifiedvnb} There are several situations where the virtual normal bundle formula simplifies dramatically.  If $n=0$, then $D=1$ in the Grothendieck ring and we have \begin{eqnarray} \label{simplifiedNvir} N^{\rm vir}_{Y/P} &=& p_!(2E + SE)-p_!(\land^2 E +S + S^\lor \land^2 E)   \quad \in K_T(Y).  \end{eqnarray}  If, furthermore, we restrict to the locus $U \subset Y$ where $S$ is a subbundle (i.e.\ where $Q$ is locally free), then we have $\land^2 E=SQ$ and $E=S+Q$. The formula becomes \begin{eqnarray} \label{bestNvir} N^{\rm vir}_{Y/P} &=& p_!(S+Q+S^2)   \quad \in K_T(Y).  \end{eqnarray}  On this locus, the line subbundle $L$ generated by $S$ coincides with $S$, which is to say: the universal stable pair is nothing but the structure sheaf of an embedded curve $Z \subset E$ (really: a flat family of embedded curves $Z \subset U \times E$) .  Using the formula \eqref{normalbundleofZ}, we recognize $S+Q+S^2$ as the class of the moving part of the normal bundle $N_{Z/U \times E}$ in the Grothendieck group (the fixed part of the normal bundle is our old friend $S^\lor Q$). \end{rem}

In the course of the proof we saw that the $T$-fixed part of \eqref{virtualnormalbundle} was given by \begin{eqnarray*} p_!(D-2+S^\lor E ) & = & p_!(D-1+S^\lor Q) \\ & = & \bR p_*(\O_D(D) \oplus \sHom(S, Q) \end{eqnarray*} in the Grothendieck ring.  Our next step is to lift this equality to the actual derived category.  We begin with the case $n=0$.

\begin{lem} \label{lem:vfc1} When $n=0$, the $T$-fixed part of \eqref{virtualnormalbundle} is isomorphic to $\bR p_* \sHom(S,Q)$.  The $T$-fixed stable pairs perfect obstruction theory gives rise to a virtual class $$[ \Quot^e E ]^{\rm vir} \in A_{1-g+d-2e}(\Quot^e E).$$ \end{lem}

\begin{proof} The sheaf $F = \O_{Y \times C}[S^\lor]$ on $Y \times E$ admits a $T$-equivariant (i.e.\ graded) locally free resolution $$\pi^* S^\lor \land^2 E^\lor \oplus \pi^* \land^2 E^\lor \to \pi^* E^\lor S^\lor \oplus \pi^* E^\lor \to \pi^* S^\lor \oplus \pi^* \O_{Y \times C}$$ whose first few graded pieces look like: $$\begin{array}{rcccl} & & & & \O_{Y \times C}   \\ & & E^\lor & \to & S^\lor \oplus E^\lor  \\ \land^2 E^\lor & \to & E^\lor S^\lor \oplus E^\lor E^\lor & \to & E^\lor S^\lor \oplus \Sym^2 E^\lor \\ S^\lor \land^2 E^\lor \oplus E^\lor \land^2 E^\lor & \to & \Sym^2 E^\lor S^\lor \oplus \Sym^2 E^\lor E^\lor & \to & \Sym^2 E^\lor S^\lor \oplus \Sym^3 E^\lor \end{array}$$ (here we write, e.g.\ $\Sym^2 E^\lor S^\lor$ for $(\Sym^2 E^\lor) \otimes S^\lor$ as opposed to $\Sym^2 (E^\lor \otimes S^\lor)$).  That is, all the rows except the first are exact except the cokernel of the rightmost map in the second row is $S^\lor$.

Appyling $\sHom( \slot, F)$ to this, we get a complex quasi-isomorphic to $\bR \sHom(F,F).$  Its graded pieces look like: $$\begin{array}{rcccl} & & & & S \land^2 E \\ &  &  SE & \to & \land^2 E \oplus \land^2 E \\ S & \to & E \oplus E & \to & S^\lor \land^2 E \\ \O_{Y \times C} \oplus \O_{Y \times C} & \to & S^\lor E \\ S^\lor \end{array}$$  From a brief inspection of the boundary maps, two things are clear: (1) The kernel of the leftmost map is just $F$, so $\sHom(F,F)=F$ (we could see this directly anyway) and (2) In grading zero, the $T$-fixed subcomplex (the penultimate row above) is cohomologically formal (quasi-isomorphic to the direct sum of its shifted cohomology sheaves).  That is, we have \begin{eqnarray*} \bR \sHom(F,F)^T & = & \O_{C \times Y} \oplus S^\lor Q [-1]  \\ & = & \sHom(F,F)^T \oplus \iota_* \sHom(S,Q)[-1] \end{eqnarray*} in $D(Y \times E)$.

Similarly, if we apply $\sHom( \slot, \O_{Y \times E})$ to the above resolution of $F$, we get a complex quasi-isomorphic to $\bR \sHom(F, \O_X)$.  Its first few graded pieces look like $$\begin{array}{rcccl} & & & & S \land^2 E \\ &  &  SE & \to & \land^2 E \oplus E^\lor S \land^2 E \\ S & \to & E \oplus E^\lor S E & \to & E^\lor \land^2 E \oplus \Sym^2 E^\lor S \land^2 E \\ \O_{Y \times C} \oplus E^\lor S & \to & E^\lor E \oplus \Sym^2 E^\lor SE & \to & \Sym^2 E^\lor \land^2 E \oplus \Sym^3 E^\lor S \land^2 E. \end{array}$$  Examining the boundary maps and doing a little linear algebra, we can see that all the rows are exact except the first two.  In particular, the $T$-fixed subcomplex (the bottom row) is exact, hence quasi-isomorphic to zero.  The map in the second row is monic with cokernel $\land^2 E$, and we see that $$ \bR \sHom(F,\O_{Y \times E}) = F \otimes \pi^* S \land^2 E [-2].$$ In fact, for any stable pair $\O_X \to F$, we always have $$\bR \sHom(F,\O_X) = \sExt^2(F,\O_X)[-2]$$ because the sheaves $\sExt^i(F,\O_X)$ vanish for $i \neq 2$ because $F$ is supported in codimension $2$ and has cohomological dimension $\leq 2$ (c.f.\ Page 10 in \cite{PT3}).

The rest of the argument is quite formal, but we include it for completeness. Applying $\bR \sHom( \slot, \O_{Y \times E})$ to the triangle $$F[-1] \to I^\bullet \to \O_{Y \times E},$$ we get a triangle $$\O_{Y \times E} \to \bR \sHom(I^\bullet , \O_X) \to \bR \sHom(F[-1] , \O_{Y \times E}) .$$ Examining the associated long exact cohomology sequence and using the established vanishings, we find: \begin{eqnarray*}  \sHom(I^\bullet, \O_{Y \times E}) &=& \O_{Y \times E} \\ \sExt^1(I^\bullet, \O_{Y \times E}) & = & \sExt^2(F,\O_X) \\ \sExt^i(I^\bullet, \O_{Y \times E}) & = & 0, \quad \quad i \neq 0,1. \end{eqnarray*}

Applying $\bR \sHom( \slot, \O_{Y \times E})$ to the same triangle, we get a triangle $$F[-1] \to \bR \sHom(I^\bullet, F)[-1] \to \bR \sHom(F,F)$$ whose long exact cohomology sequence looks like $$\begin{array}{rccccl} & 0 & \to & \sExt^{-1}(I^\bullet,F) & \to & \sHom(F,F) \\  \to & F & \to & \sHom(I^\bullet,F) & \to & \sExt^1(F,F) \\ \to & 0 & \to & \sExt^1(I^\bullet,F) & \to & \sExt^2(F,F). \end{array}$$  The map $\sHom(F,F) \to F$ is the one induced by $s : \O_{Y \times E} \to F$ and is given by $\phi \mapsto \phi(s(1))$.  It is easily seen to be an isomorphism since this $F$ is in fact a sheaf of $\O_{Y \times E}$ \emph{algebras} (Remark~\ref{rem:Fasacurve}), hence we can use the multiplication by any local section.  We conclude: \begin{eqnarray*} \sHom(I^\bullet,F) &=& \sExt^1(F,F) \\ \sExt^1(I^\bullet,F) &=& \sExt^2(F,F) \\ \sExt^i(I^\bullet,F) &=& 0, \quad \quad i \neq 0,1. \end{eqnarray*}

Finally we apply $\bR \sHom( \slot, I^\bullet)$ to the same triangle to get the following long exact sequence: $$\begin{array}{rccccl}  & 0 & \to & \O_{Y \times E} & \to & \O_{Y \times E} \\ \to & \sExt^1(F,F) & \to & \sExt^1(I^\bullet,I^\bullet) & \to & \sExt^1(I^\bullet,\O_{Y \times E}) \\ \to & \sExt^2(F,F) & \to & \sExt^2(I^\bullet,I^\bullet) & \to & 0 \end{array}$$  The map in the first row is an isomorphism (c.f.\ Lemma~1.20 in \cite{PT}) and we have already established the vanishings $$\sExt^1(I^\bullet,\O_{Y \times E})^T = \sExt^2(F,\O_{Y \times E})^T = 0$$ and $\sExt^2(F,F)^T=0$, so we have: $$ \iota_* \sHom(S,Q) = \sExt^1(F,F)^T = \sExt^1(I^\bullet,I^\bullet)^T .$$  Since scalar multiplication $\O_{Y \times E} \to \sHom(I^\bullet,I^\bullet)$ is an isomorphism and the trace map splits this off of $\bR \sHom(I^\bullet,I^\bullet)$ we may in fact view the above isomorphism as an isomorphism $$ \iota_* \sHom(S,Q) = \bR \sHom(I^\bullet,I^\bullet)^T_0[1] $$ in $D(Y \times E)$.  Applying $\bR \pi_*$ and noting that $\bR \pi_* \iota_* = \bR p_*$ completes the proof.
 
\end{proof}

\begin{rem} One can check that this virtual class on $\Quot^e E$ is the ``well known" \cite{CFK}, \cite{MO} class discussed in Section~\ref{section:virtualclass}. \end{rem}  

We can use the same technique to handle the general case.

\begin{thm} \label{thm:vfc} The $T$-fixed part of the stable pairs POT \eqref{virtualnormalbundle} on $Y$ is given by $$ \bR p_* \sHom(S,Q) \oplus T \Sym^d C,$$ hence the corresponding virtual fundamental class is $$[ \Quot^e E ]^{\rm vir} \times [\Sym^d C],$$ where $[\Quot^e E]^{\rm vir}$ is the virtual class on $\Quot^e E$ as in Lemma~\ref{lem:vfc1}.  \end{thm}

\begin{proof} According to the proof of \eqref{thm:Tfixedstablepairs}, the universal stable pair on $Y \times E$ is $$I^\bullet = [\O_{Y \times E} \to F ],$$ where $F=F'(D)$ is obtained from $F'=\O_{Y \times C}[S^\lor]$ by tensoring with the section $\O_{Y \times E} \to \O_{Y \times E}(D)$ determined by the Cartier divisor $D \subset Y \times E$.  Note $F'$ is pulled back from $\Quot^e E \times C$ and $D$ is pulled back from $\Sym^d C \times C.$  In particular, $\O_{Y \times E}(D)$ is locally free, so we have equalities: \begin{eqnarray*} \bR \sHom(F'(D),F'(D)) &=& \bR \sHom(F',F') \\ \bR \sHom( F'(D), \O_{Y \times E} ) &=& \bR \sHom(F', \O_{Y \times E})(-D), \end{eqnarray*} and so forth (we have used the Grothendieck groups analogues in the proof of Lemma~\ref{lem:virtualnormalbundle}).  In particular, we can see that $$\bR \sHom(F'(D), \O_{Y \times E})^T = 0,$$ just as we did in the proof of Lemma~\ref{lem:vfc1}.

The only departure from the proof of Lemma~\ref{lem:vfc1} occurs when we apply $\bR \sHom( \slot, \O_{Y \times E})$ to the triangle $$F'(D)[-1] \to I^\bullet \to \O_{Y \times E}.$$  This time the long exact cohomology sequence looks like  $$\begin{array}{rccccl} & 0 & \to & \sExt^{-1}(I^\bullet,F'(D)) & \to & \sHom(F',F') \\  \to & F'(D) & \to & \sHom(I^\bullet,F'(D)) & \to & \sExt^1(F',F') \\ \to & 0 & \to & \sExt^1(I^\bullet,F'(D)) & \to & \sExt^2(F',F'). \end{array}$$  As in the proof of Lemma~\ref{lem:vfc1}, we have $\sHom(F',F')=F',$ and the map $\sHom(F',F') \to F'(D)$ is the obvious one.  The $T$ fixed part of this map is the usual map $\O_{Y \times C} \to \O_{Y \times C}(D)$ whose cokernel is $\O_D(D)$, so the $T$-fixed part of the middle row yields a SES $$0 \to \O_D(D) \to \sExt^1(I^\bullet, F'(D))^T \to \sExt^1(F',F')^T \to 0.$$  This sequence splits using the identification $\sExt^1(F',F')^T = \sExt^1(I^\bullet, F')^T$ from the proof of \eqref{lem:vfc1} and the natural inclusion $$ \sExt^1(I^\bullet, F')^T \to \sExt^1(I^\bullet, F')^T(D) = \sExt^1(I^\bullet, F'(D))^T.$$

Exactly as in the proof of \eqref{lem:vfc1} we obtain an identification $$ \iota_* \sHom(S,Q) \oplus \O_D(D) = \bR \sHom(I^\bullet, I^\bullet)^T_0[1].$$  Finally, we use the standard formula \begin{eqnarray*} T \Sym^d C & = & p_* \sHom(I_D, \O_D) \\ &=& p_* \sHom( \O_{\Sym^d C \times C}(-D), D) \\ & = & \bR p_* \O_D(D) \end{eqnarray*} for the tangent space of the (smooth) Hilbert scheme $\Sym^d C$. \end{proof}

\section{Computations} \label{section:computations}

Consider the expression for the virtual normal bundle of $Y=\Quot^e E \times \Sym^n C$ in $P=P_{2-2g-e+2n}(E,2)$ given in Lemma~\ref{lem:virtualnormalbundle}.  It is clear from GRR that $e_T(-N_{Y/P}^{\rm vir})$ can be expressed in terms of tautological classes on $Y$; therefore \begin{eqnarray} \label{localcontribution} \int_{[Y]^{\rm vir}} e_T(-N^{\rm vir}_{Y/P}) & \in & \ZZ[t,t^{-1}] \end{eqnarray} can be computed using Theorem~\ref{thm:virtualintersectionnumbers}.  The purpose of this section is to write out an explicit formula for $e_T(-N_{Y/P}^{\rm vir})$ and calculate \eqref{localcontribution} in various cases.  A full reconciliation of our computations with the prediction of Equation~\ref{ZPT} poses combinatorial difficulties not addressed herein.   On the other hand, the formula derived below is simple enough that it is feasible to compute any particular invariant $P_{n,2}(d)$ (with minor computer assistance).

\subsection*{Notation}  Let $\pi_1,\pi_2$ be the projections from $Y$ to $\Quot^e E$, $\Sym^n C$, respectively.  Let $a_1 \in \H^2(Y)$ be the pullback of the $a$ class on $\Quot^e E$ via $\pi_1$ (Section~\ref{section:tautologicalclasses}) and let $b_{1,i} \in \H^1(Y)$ be the pullback of the $b_i$ class on $\Quot^e E$ via $\pi_1$.  Similarly, let $a_2 \in \H^2(Y)$ be the pullback of the $a$ class (the coefficient of $\eta$ in the K\"unneth decomposition of the universal divisor $D \subset \Sym^n C \times C$) on $\Sym^n C$ via $\pi_2$ and let $b_{2,i} \in \H^1(Y)$ be the pullback of the $b_i$ class on $\Sym^n C$.  Set $\theta_j := \sum_{i=1}^g b_{j,i}b_{j,g+i} \in \H^2(Y)$, so $\theta_j$ is the pullback of the $\theta$ class on the Jacobian under the composition of $\pi_j$ and the Abel-Jacobi map.

It is convenient to set $$B := \sum_{i=1}^g b_{1,i} b_{2,g+i}-b_{1,g+i} b_{2,i} $$ so that \begin{eqnarray} \label{crossterm}  \left ( \sum_{i=1}^{2g} b_{1,i} \delta_i \right ) \left ( \sum_{i=1}^{2g} b_{2,i} \delta_i \right ) &=& -B \eta \end{eqnarray} in $\H^*(Y \times C)$.  Write $\approx$ for equality in $\H^*(\Jac C)$ modulo odd monomials (c.f.\ Section~\ref{section:tautologicalclasses}).  Then we have  $$ B^{2l} \approx  (-1)^l \begin{pmatrix} 2l \\ l \end{pmatrix} l! \sum_{1 \leq i_1 < \cdots < i_l \leq g} b_{1,i_1}b_{1,g+i_1} \cdots b_{1,i_l}b_{1,g+i_l} b_{2,i_1} b_{2,g+i_1} \cdots b_{2,i_l} b_{2,g+i_l}.$$ For distinct $i_1,\dots,i_m \in \{ 1 ,\dots, g \}$, Theorem~\ref{thm:virtualintersectionnumbers} yields \begin{eqnarray} \label{formula1} \int_{[\Quot^e E]^{\rm vir}} a^{1-g+d-2e-m} b_{i_1}b_{g+i_1} \cdots b_{i_m}b_{g+i_m} &=& 2^{g-m} \\ \nonumber \int_{\Sym^n C} a^{n-m} b_{i_1} b_{g+i_1} \cdots b_{i_m} b_{g+i_m} &=& 1.\end{eqnarray}  Since \begin{eqnarray*} \theta^j &=& j! \sum_{1\leq i_1 < \cdots < i_j \leq g} b_{i_1}b_{g+i_1} \cdots b_{i_j}b_{g+i_j} \end{eqnarray*} and the product of an even and an odd monomial is odd it follows that $\theta_1^j \theta_2^k B^{2l}$ is equal (modulo odd monomials in the $b_{1,i}$ and $b_{2,i}$) to $$ (-1)^l \begin{pmatrix} 2l \\ l \end{pmatrix} j!k!l!$$ times a sum of $$ \begin{pmatrix} g \\ l \end{pmatrix} \begin{pmatrix} g-l \\ j \end{pmatrix} \begin{pmatrix} g-l \\ k \end{pmatrix} $$ terms of the form $$b_{1,i_1}b_{1,g+i_1} \cdots b_{1,i_{j+l}} b_{1,g+i_{j+l}} b_{2,i_1} b_{2,g+i_1} \cdots b_{2,i_{k+l}} b_{2,g+i_{k+l}} .$$  Since $[Y]^{\rm vir} = [\Quot^e E]^{\rm vir} \times [\Sym^n C]$, the formula \begin{eqnarray} \label{mainformula} \int_{[Y]^{\rm vir}} a_1^{1-g+d-2e-j-l} a_2^{n-k-l} \theta_1^j \theta_2^k B^{2l} &=& (-1)^l \begin{pmatrix} 2l \\ l \end{pmatrix} \frac{g!(g-l)!2^{g-j-l}}{(g-j-l)!(g-k-l)!} \end{eqnarray} therefore follows from \eqref{formula1} after slight simplifications.

In this notation, we have \begin{eqnarray*} \ch S^\lor &=& e^{a_1-t}(1-\sum_{i=1}^{2g}b_{1,i} \delta_i - e \eta - \theta_1 \eta) \\ \ch D &=& e^{a_2}(1-\sum_{i=1}^{2g} b_{2,i} \delta_i + n \eta - \theta_2 \eta) \end{eqnarray*} in $\H^*_T(Y \times C)$.  To calculate, for example, $e_T(p_!S^\lor D)$, we first calculate $$\ch S^\lor D = e^{a_1+a_2-t} \left ( 1+\sum_{i=1}^{2g}(b_{1,i}+b_{2,i})\delta_i +(n-e-\theta_1-\theta_2-B)\eta \right )$$ using \eqref{crossterm}, then by GRR for $p :Y \times C \to Y$, we compute $$\ch p_!S^\lor D = e^{a_1+a_2-t}  ( 1-g+n-e-\theta_1-\theta_2-B).$$  Evidently $p_!S^\lor D$ has the same Chern character as the product of a line bundle with $c_1=a_1+a_2-t$ and a (virtual) vector bundle of rank $1-g+n-e$ and total Chern class $e^{-\theta_1-\theta_2-B}$, so from the usual formula for the Euler class of such a product, we obtain: $$ e_T(p_! S^\lor D) = (a_1+a_2-t)^{1-g+n-e} \exp \left ( \frac{-\theta_1-\theta_2-B}{a_1+a_2-t} \right ). $$  Similar calculations yield \begin{eqnarray*} \label{eTNvir}  e_T(-N_{Y/P}^{\rm vir}) &=& \frac{e_T(p_!(S^\lor+2\land^2 E+S \land^2 E+S+S^\lor \land^2 E))}{e_T(p_!(S^\lor D+D^\lor \land^2 E+D^\lor S \land^2 E+2E +SE))} \\  &=& \frac{(a_1-t)^{1-g-e}(2t)^{2-2g+2d}(3t-a_1)^{1-g+e+d}}{(a_1+a_2-t)^{1-g+n-e}(2t-a_2)^{1-g+d-n}(3t-a_1-a_2)^{1-g+d+e-n}} \\  & & \cdot \frac{(t-a_1)^{1-g+e}(t+a_1)^{1-g+d-e}}{(t)^{4-4g+2d}(2t-a_1)^{2-2g+d+2e}} \\  & & \cdot \exp \left ( \frac{-\theta_1}{a_1-t} \right ) \exp \left ( \frac{-\theta_1}{3t-a_1} \right ) \exp \left ( \frac{-\theta_1}{t-a_1} \right ) \exp \left ( \frac{-\theta_1}{t+a_1} \right ) \\  & & \cdot \exp \left ( \frac{\theta_1+\theta_2+B}{a_1+a_2-t} \right ) \exp \left ( \frac{\theta_2}{2t-a_2} \right ) \exp \left ( \frac{\theta_1+\theta_2+B}{3t-a_1-a_2} \right )  \exp \left ( \frac{2 \theta_1}{2t-a_1} \right ). \end{eqnarray*}  Expanding this out and integrating over $[Y]^{\rm vir}$ using \eqref{mainformula} presents no particular difficulty in any given case, though a simple general formula is difficult to obtain.

When $n=0$, $Y=\Quot^e E$, and this formula simplifies to \begin{eqnarray} \label{eTNvir} e_T(-N^{\rm vir}_{Y/P}) &=& \frac{e_T(p_!(\land^2 E +S+S^\lor \land^2 E))}{e_T(p_!(2E+SE))} \\ \nonumber &=& \frac{(2t)^{1-g+d}(t-a)^{1-g+e}(t+a)^{1-g+d-e}}{(t)^{4-4g+2d}(2t-a)^{2-2g+d+2e}} \\ \nonumber & & \cdot \exp \left ( \frac{-\theta}{t-a} \right ) \exp \left ( \frac{-\theta}{t+a} \right ) \exp \left ( \frac{2 \theta}{2t-a} \right ) \\ \nonumber &=& 2^{g-2e-1}t^{3g-3-d-2e}  -2^{g-2e-1}t^{3g-4-d-2e}\theta \\ \nonumber & & +2^{g-2e-2}t^{3g-4-d-2e}(2-2g+3d-2e)a + \dots \end{eqnarray} where the $\dots$ are terms in $R^{>2}(\Quot^e E)$.

\subsection{Expected dimension one} Suppose $E$ is a rank $2$ bundle on $C$ whose degree $d$ has the same parity as the genus $g$ of $C$.  Define an integer $e$ by setting $d-g=2e.$  Then the smallest power of $q$ appearing in the Laurent expansion of $Z_2^{\rm PT}(d)$ occurs when $i=e$ in \eqref{ZPT}.  It is given by $t^{4g-4-2d}2^{3g-1-d}(d+1-g)q^{2-d-g},$ so the DT=PT invariant of $E$ in minimal Euler characteristic is \begin{eqnarray} \label{example1} P_{2-d-g,2}(d) &=& t^{4g-4-2d}2^{3g-1-d}.\end{eqnarray}  On the other hand, Theorem~\ref{thm:Tfixedstablepairs} provides an identification $$P_{2-d-g}(E,2)^T = \Quot^e E $$ (after possibly throwing away fixed components with negative expected dimension).  The formulae \begin{eqnarray*} \int_{[\Quot^e E]^{\rm vir}} a &=& 2^g \\ \int_{[\Quot^e E]^{\rm vir}} \theta &=& g2^{g-1} \end{eqnarray*} are obtained by taking $k=g,g-1$ (respectively) in \eqref{powersofatheta}.  Using \eqref{eTNvir}, we then compute \begin{eqnarray*} P_{2-d-g,2}(d) &=& \int_{[P_{2-d-g}(E,2)^T]^{\rm vir}} e_T(-N^{\rm vir}) \\ &=&  -2^{g-2e-1}t^{3g-4-d-2e} \int_{ [\Quot^e E]^{\rm vir} } \theta \\ & & +2^{g-2e-2}t^{3g-4-d-2e}(2-2g+3d-2e) \int_{ [\Quot^e E]^{\rm vir} } a \\ &=& -2^{2g-d-1}t^{4g-4-2d}(g2^{g-1}) +2^{2g-d-2}t^{4g-4-2d}(2-g+2d)(2^g) \\ &=& t^{4g-4-2d}2^{3g-1-d} \end{eqnarray*} in agreement with the formula \eqref{example1} obtained from the GW/PT correspondence and the equivalence of DT and PT invariants in minimal Euler characteristic.

\subsection{Target genus zero}  When $C=\PP^1$ the Jacobian is trivial so the $\Quot$ schemes are just projective spaces.  We can identify the virtual fundamental class on the $\Quot$ scheme and explicitly compute the PT invariants from first principles (i.e.\ without using Theorem~\ref{thm:virtualintersectionnumbers} or the GRR calculation above).  Consider the case where $C=\PP^1$ and $E= \O(d_1) \oplus \O(d_2)$, with $d=d_1+d_2$.

If we make the identification $$\Quot^e E \times \Sym^n \PP^1 \times \PP^1 \cong \PP^N \times \PP^n \times \PP^1$$ (so $p$ is the projection on the first two factors), then the bundles appearing in the virtual normal bundle formula of Lemma~\ref{lem:virtualnormalbundle} are tensor products and duals of the following bundles: \begin{eqnarray*} S &=& \O(-1,0,e)_t \\ E &=& \O(0,0,d_1)_t \oplus \O(0,0,d_2)_t \\ \land^2 E & = & \O(0,0,d)_{2t} \\ D &=& \O(0,1,n)_0. \end{eqnarray*}  We have used subscripts to keep track of the weight of the $T$ action under these identifications.  (The description of the universal bundle on $\Sym^n \PP^1 \times \PP^1$ is elementary and we discussed the universal bundle $S$ on $\Quot^e E \times \PP^1$ in Section~\ref{section:quotschemes}.)  

\begin{lem} Assume $1+d-2e \geq 0$.  If $E$ is balanced, the virtual fundamental class on $\Quot^e E$ is its usual fundamental class.  In any case, $\Quot^e E$ is a projective space of dimension at least $1+d-2e$ and the virtual class is just the fundamental class of a linearly embedded projective space of the expected dimension $1+d-2e$. \end{lem}

\begin{proof} As discussed in \eqref{section:quotschemes}, $$\Quot^e E = \PP( \H^0(\PP^1,\O(d_1-e)) \oplus \H^0(\PP^1,\O(d_2-e))).$$  We have $h^0(\PP^1,\O(d_1-e)=1+d_1-e$ and $h^0(\PP^1,\O(d_2-e))=1+d_2-e$, so the moduli space is smooth of the expected dimension unless one of the $1+d_i-e$ is negative and the formula is invalid.  They can't \emph{both} be negative because we are assuming $$d_1-e+d_2-e \geq -1,$$ and balancing will also ensure both expressions are nonnegative.  Say $1+d_1-e$ is negative, so we have: 

\begin{tabular}{lcl} Actual Dimension of $\Quot^e E$ & = & $d_2-e$ \\ Expected Dimension of $\Quot^e E$ &=& $1+d_1+d_2-2e$ \\ Excess Dimension & = & $e-d_1-1.$ \end{tabular} 

In this situation, $\O(e)$ admits no nonzero map to $\O(d_1)$, so every point of the $\Quot$ scheme will look like $$0 \to \O(e) \to \O(d_1) \oplus \O(d_2) \to \O(d_1) \oplus T \to 0,$$ where $T$ is a torsion sheaf of degree $d_2-e$.  The torsion sheaf doesn't contribute to the higher direct image under $$p : \Quot^e E \times \PP^1 \to \Quot^e E,$$ and we compute \begin{eqnarray*} \R^1 p_* S^\lor Q & = & \R^1p_*( \O(-1,e)^\lor \otimes \O(0,d_1) ) \\ & = & \R^1 p_* \O(1,d_1-e) \\ & = & \O(1) \otimes_\CC \H^1(\PP^1, \O(d_1-e)) \\ & = & \O(1)^{e-d_1-1} \end{eqnarray*} using the projection formula.  Since the virtual class is $e( \O(1)^{e-d_1-1}) \cap [\Quot^e E]$ the proof is complete.  \end{proof}

Let $a_i \in \H^2(\PP^N \times \PP^n)$ be the pullback of the hyperplane class along $\pi_i$ (notice that the hyperplane class is the $a$ class when the rank one Quot scheme is identified with a projective space).  Because of the lemma, the actual value of $$N = \dim \Hom(\O(e),\O(d_1)\oplus \O(d_2)) - 1 $$ will be irrelevant.  When we integrate over $[\PP^N]^{\rm vir}$ it will be as if $$N = 1+d-2e.$$ Pushing forward over the $\PP^1$ factor is easy; the projection formula implies $$p_! \O(a,b,c) = (c+1) \O(a,b)$$ in the Grothendieck ring.  For example, we find $$e_T(p_! S) = (t-a_1)^{e+1}.$$

Since we now know how to work out the $T$ equivariant Euler class of the virtual normal bundle, the computation of the PT invariant is is now just a matter of putting the pieces together: \begin{eqnarray*} P_{2+m,2}(d) & = & \int_{P_{2+m}(E,2)^T} e_T(-N^{\rm vir}) \\ & = & \sum_{\stackrel{2n-e=m}{1+d-2e \geq 0}} \int_{[\PP^N]^{\rm vir} \times \PP^n} \frac{e_T(p_!(S^\lor+2 \land^2 E +S \land^2 E + S + S^\lor \land^2 E))}{e_T(p_!(S^\lor D + D^\lor \land^2 E + D^\lor S \land^2 E + 2E + SE))} \\ & = & \sum_{\stackrel{2n-e=m}{1+d-2e \geq 0}} C(d,e,n), \end{eqnarray*} where we define $C(d,e,n)$ to be the coefficient of $a_1^{1+d-2e}a_2^n$ in the expression $$\frac{(a_1-t)^{1-e}(2t)^{2d+2}(3t-a_1)^{d+e+1}(t-a_1)^{e+1}(a_1+t)^{1-d-e}}{(a_1+a_2+t)^{1+n-e}(2t-a_2)^{1+d-n}(3t-a_1-a_2)^{1+d+e-n}(t)^{2d+4}(2t-a_1)^{d+2e+2}}.$$ 

Presumably it is possible to completely reconcile this with Equation~\ref{ZPT}, but we will content ourselves here with a few explicit computations.

\subsection{Local $\PP^1$}  \label{section:localP1} Consider the bundle $E = \O(-1)^{\oplus 2}$ on $\PP^1$.  The stable pairs spaces $P_m(E,2)$ are easily seen to be compact, and each has expected dimension zero, so the residue invariants are bonafide integrals taking values in $\ZZ$.  The total space $E$ is Calabi-Yau, so the invariants can be computed as weighted Euler characteristics of the moduli space, weighted by Behrend's constructible function (Lemma~1.3 in \cite{PT3}, \cite{Beh}). The pairs space $P_4(E,2)$, for example, is discussed extensively in Section~4.1 of \cite{PT}.  It can be described as the closed subscheme $$P_4(E,2) \hookrightarrow \Spec_{\PP^3} \O_{\PP^3}[\O_{\PP^3}(2)] $$ of the first infinitesimal neighborhood of $\PP^3$ in $\O(-2)$ determined by the ideal generated by sections of $\O_{\PP^3}(2)$ vanishing along a quadric $\PP^1 \times \PP^1 \subseteq \PP^3$.  Locally, it looks like a product of $$Z = \Spec \CC[x,y]/(y^2,xy)$$ and $\AA^2$.  The normal cone of $Z \subset \AA^2_{x,y}$ has two irreducible components, one lying over the origin with length $2$ at its generic point and one generically smooth component surjecting to the topological space $|Z|=|\AA^1|$, so the characteristic cycle (\cite{Beh}, 1.1) of $Z$ is $\mathfrak{c}_Z=-[\AA^1]+2[0]$.  These cycles are smooth, so the local Euler obstruction construction of Section~1.2 in \cite{Beh} simply yields their characteristic functions, so Behrend's function is $-1$ on the smooth locus and $+1$ at the origin.  Since Behrend's function $\nu$ is local and respects products, we conclude that it is given by $1$ on the quadric $\PP^1 \times \PP^1 \subset \PP^3$ and $-1$ away from the quadric, so the pairs invariant is $$P_{4,2}(-2) = -1 \cdot \chi(\PP^3 \setminus \PP^1 \times \PP^1) + 1 \cdot \chi(\PP^1 \times \PP^1) = -1 \cdot 0+1 \cdot 4=4.$$  

From our point of view, we have $$P_4(E,2)^T = \Quot^{-2} E = \PP \Hom( \O(-2), E) \cong \PP^3.$$ The universal sequence $$0 \to S \to \pi_2^* E \to Q \to 0$$ on $\Quot^{-2} E \times \PP^1$ is identified with the sequence $$0 \to \O(-1,-2) \to \O(0,-1)^{\oplus 2} \to Q \to 0$$ on $\PP^3 \times \PP^1$.  The quotient sheaf $Q$ restricts to $\O$ on $\{ P \} \times \PP^1$ for a generic $P \in \PP^3$, but $Q$ restricts to the direct sum of $\O(-1)$ and the structure sheaf of a point when $P$ is in the stratum $$\Quot^{-1} E \times \Sym^1 \PP^1 \hookrightarrow \Quot^{-2} E $$ discussed in Section~\ref{section:quotschemes}.  Of course, $\Quot^{-1} E = \PP \Hom( \O(-1),E) \cong \PP^1$ and this stratum is exactly the quadric $\PP^1 \times \PP^1 \hookrightarrow \PP^3$ along which the full stable pairs space has a $Z \times \AA^2$ singularity.  We can compute the PT invariant explicity using the method explained at the beginning of this section: \begin{eqnarray*} P_{4,2}(-2) & = & \int_{[P_4(E,2)^T]^{\rm vir}} e_T(-N^{\rm vir}) \\ & = & \int_{[\Quot^{-2} E]^{\rm vir}} e_T(-N^{\rm vir}) \\ & = & \int_{\PP^3} e_T(p_!(S^\lor \land^2 E - 2E + S + \land^2 E - SE)) \\ & = & \int_{\PP^3} \frac{e_T(p_* S^\lor \land^2 E)e_T(R^1 p_* SE)}{e_T(R^1 p_*S)e_T(R^1p_* \land^2 E)} \\ & = & \int_{\PP^3} \frac{e_T( \O(1)_t)e_T(\O(-1)_{2t})^4}{e_T(\O(-1)_t) e_T(\O_{2t})} \\ & = & \int_{\PP^3} \frac{(a+t)(2t-a)^4}{(t-a)(2t)} \\ & = & \int_{\PP^3} 4a^3 - 4a^2t + 8 t^3 \\ & = & 4. \end{eqnarray*}

According to \eqref{ZPT}, the generating function for degree $2$ stable pairs invariants of local $\PP^1$ is \begin{eqnarray*} Z_2^{\rm PT}(-2) &=& \sum_m P_{m,2}(\O(-1)^{\oplus 2},2)q^m \\ &=& -2q^3(1+q)^{-4}(1-q)^{-2} \\ &=& -2q^3(1-(-q)^2)^{-2}(1-q^2)^{-2} \\ & = & -2q^3 \left( \sum_{d=0}^{\infty} (d+1)(-q)^d \right ) \left ( \sum_{e=0}^{\infty} (e+1)(-q)^{2e} \right )  \\ & = & 2 \sum_{m=3}^{\infty} \; \sum_{d+2e=m-3} (d+1)(e+1)(-q)^m \\ &=& -2q^3+4q^4-10q^5+16q^6-28q^7+\cdots \end{eqnarray*}

Our computations yield \begin{eqnarray*} Z_2^{\rm PT}(-2) &=& C(-2,-1,0)q^3 + C(-2,-2,0)q^4 \\ & & + (C(-2,-3,-0)+C(-2,-1,-1))q^5 \\ & & +(C(-2,-4,0)+C(-2,-2,1))q^6  \\ & & +(C(-2,-5,0)+C(-2,-3,1)+C(-2,-1,2))q^7 +\cdots \\ & = & -2q^3+4q^4+(18-28)q^5+(424-408)q^6 +(7750-8404+626)q^7 + \cdots \end{eqnarray*} in complete agreement.

\end{document}